\documentclass[12pt]{amsart}

\usepackage[margin = 1in]{geometry}
\usepackage{amsfonts, amsmath,amscd, amssymb,
latexsym, mathrsfs, mathtools, 
slashed, stmaryrd, verbatim, wasysym, bbm, url}
\usepackage[usenames]{xcolor}
\usepackage{enumerate}
\usepackage[shortlabels]{enumitem}
\usepackage[utf8]{inputenc}    

\usepackage{hyperref}
\hypersetup{
  colorlinks   = true,    % Colours links instead of ugly boxes
  urlcolor     = blue,    % Colour for external hyperlinks
  linkcolor    = blue,    % Colour of internal links
  citecolor    = red      % Colour of citations
}
      
\usepackage[backend=bibtex,bibstyle=authoryear,citestyle=alphabetic,firstinits=true,doi=false,isbn=false,url=false]{biblatex}
\DeclareFieldFormat{labelalphawidth}{\mkbibbrackets{#1}}

\defbibenvironment{bibliography}
  {\list
     {\printtext[labelalphawidth]{%
        \printfield{prefixnumber}%
    \printfield{labelalpha}%
        \printfield{extraalpha}}}
     {\setlength{\labelwidth}{\labelalphawidth}%
      \setlength{\leftmargin}{\labelwidth}%
      \setlength{\labelsep}{\biblabelsep}%
      \addtolength{\leftmargin}{\labelsep}%
      \setlength{\itemsep}{\bibitemsep}%
      \setlength{\parsep}{\bibparsep}}%
      }
  {\endlist}
  {\item}
\newtheorem*{acknowledgements}{Acknowledgements}
\newtheorem*{theorem*}{Theorem}
\newtheorem{theorem}{Theorem}[section]
\newtheorem{corollary}{Corollary}[section]

\newtheorem{example}{Example}[section]
\newtheorem{lemma}[corollary]{Lemma}
\newtheorem{proposition}[corollary]{Proposition}
\theoremstyle{definition}
\newtheorem{remark}[example]{Remark}

\numberwithin{equation}{section}
\let\oldsqrt\sqrt
\def\sqrt{\mathpalette\DHLhksqrt}
\def\DHLhksqrt#1#2{%
\setbox0=\hbox{$#1\oldsqrt{#2\,}$}\dimen0=\ht0
\advance\dimen0-0.2\ht0
\setbox2=\hbox{\vrule height\ht0 depth -\dimen0}%
{\box0\lower0.4pt\box2}}

\usepackage{verbatim} %% code from mathabx.sty and mathabx.dcl
\DeclareFontFamily{U}{mathx}{\hyphenchar\font45}
\DeclareFontShape{U}{mathx}{m}{n}{
      <5> <6> <7> <8> <9> <10>
      <10.95> <12> <14.4> <17.28> <20.74> <24.88>
      mathx10
      }{}
\DeclareSymbolFont{mathx}{U}{mathx}{m}{n}
\DeclareFontSubstitution{U}{mathx}{m}{n}
\DeclareMathAccent{\widecheck}{0}{mathx}{"71}

\renewcommand{\tilde}{\widetilde}
\renewcommand{\bar}{\overline}

\newcommand\eps\varepsilon
\renewcommand\epsilon\varepsilon

\newcommand{\abs}[1]{\left\lvert #1 \right\rvert}
\newcommand{\smallabs}[1]{\lvert #1 \rvert}

\newcommand{\norm}[1]{\lVert #1 \rVert}
\newcommand\inner[1]{\langle #1 \rangle}

\newcommand\dist{\operatorname{dist}}

\newcommand\Mand{\text{ and }}

\newcommand\Mfor{\text{ for }}

\newcommand\Mwith{\text{ with }}

\newcommand\paperintro%
        {%
         }
\newcommand\paperbody%
        {%
         }

\newcommand\bbE{\mathbb{E}}

\newcommand\bbG{\mathbb{G}}
\newcommand\bbH{\mathbb{H}}

\newcommand\bbL{\mathbb{L}}

\newcommand\bbN{\mathbb{N}}

\newcommand\bbR{\mathbb{R}}

\newcommand\bbZ{\mathbb{Z}}

\newcommand\cG{\mathcal{G}}

\newcommand\cL{\mathcal{L}}

\newcommand\cO{\mathcal{O}}
\newcommand\cP{\mathcal{P}}

\newcommand\cS{\mathcal{S}}

\newcommand\mf[1]{\mathfrak{ #1}}

\DeclareMathAlphabet{\mathpzc}{OT1}{pzc}{m}{it}

\newcommand{\sbs}{\subset}

\renewcommand{\bf}{\textbf}

\usepackage[refpage]{nomencl}

\makenomenclature

\usepackage{tikz, tikz-cd, tkz-euclide}
\usetikzlibrary{calc,positioning,intersections,quotes,decorations.markings}
\makeatletter
\def\@tocline#1#2#3#4#5#6#7{\relax
  \ifnum #1>\c@tocdepth % then omitth
  \else
    \par \addpenalty\@secpenalty\addvspace{#2}%
    \begingroup \hyphenpenalty\@M
    \@ifempty{#4}{%
      \@tempdima\csname r@tocindent\number#1\endcsname\relax
    }{%
      \@tempdima#4\relax
    }%
    \parindent\z@ \leftskip#3\relax \advance\leftskip\@tempdima\relax
    \rightskip\@pnumwidth plus4em \parfillskip-\@pnumwidth
    #5\leavevmode\hskip-\@tempdima
      \ifcase #1
       \or\or \hskip 1em \or \hskip 2em \else \hskip 3em \fi%
      #6\nobreak\relax
    \hfill\hbox to\@pnumwidth{\@tocpagenum{#7}}\par% <---- \dotfill -> \hfill
    \nobreak
    \endgroup
  \fi}
\makeatother
\def\annu#1{_{% 
  \vbox{\hrule height .2pt 
    \kern 1pt 
    \hbox{$\scriptstyle {#1}\kern 1pt$}% 
  }\kern-.05pt 
  \vrule width .2pt 
}}

\allowdisplaybreaks

\title[Spectral gaps in infinite dimension]{Spectral gap inequalities on nilpotent Lie groups in infinite dimensions}

\author{Esther Bou Dagher}
\address{Department of Mathematics, Imperial College London, 180 Queen’s
Gate, London, SW7 2AZ, United Kingdom}
\email{esther.bou-dagher17@imperial.ac.uk}

\author{Yaozhong Qiu}
\address{Department of Mathematics, Imperial College London, 180 Queen’s
Gate, London, SW7 2AZ, United Kingdom}
\email{y.qiu20@imperial.ac.uk}

\author{Boguslaw Zegarlinski}
\address{Department of Mathematics, Imperial College London, 180 Queen’s
Gate, London, SW7 2AZ, United Kingdom}
\email{bgslw2zeg2022@gmx.fr}

\author{Mengchun Zhang}
\address{Department of Mathematics, Imperial College London, 180 Queen’s
Gate, London, SW7 2AZ, United Kingdom}
\email{mengchun.zhang15@imperial.ac.uk}

\allowdisplaybreaks

\begin{document}

\begin{abstract}
We develop a general framework for spectral gap inequalities for Gibbs measures on infinite dimensional spin spaces over nilpotent Lie groups in terms of weak $U$-bounds and weak single-site spectral gap inequalities. We then provide sufficient conditions on the local specification and give examples of measures constructed using the Kaplan norm and generalising a few results for the Carnot-Carath\'eodory distance on the Heisenberg group.
\end{abstract}

\maketitle

\tableofcontents

%-------------------------
%       Main Body      
%-------------------------

\section{Introduction}
In this paper we provide a framework for the study of spectral gap inequalities in infinite dimensions that is well-adapted to the case where the ambient space is a nilpotent Lie group $\bbG$, and in particular when $\bbG = \bbH$ is the Heisenberg group. What we mean by a spectral gap inequality includes the Poincar\'e inequality
\begin{equation}\label{intro:spectralgap}
\int_X \abs{f - \int_X f}^2 d\mu \leq C\int_X \abs{\nabla f}^2d\mu
\end{equation}
for $(X, \Sigma, \mu)$ a smooth probability measure space, $\nabla$ a smooth (sub-)gradient, and with a constant $C > 0$ independent of $f$ belonging to a space with enough regularity conditions to ensure the right hand side makes sense. This inequality appears in the theory of Markov semigroups which play an important role in the study of Markov processes and stochastic differential equations. The infinite-dimensional part of our analysis enters in that our underlying space $X$ shall be given as (a suitable subspace of) the infinite product space $\Omega \vcentcolon= \smash{\mathbb{G}^{\bbZ^D}}$ for $\bbZ^D$ the $D$-dimensional integer lattice. 
%, that is countably infinitely many copies of the   group $\bbG$ arranged in the form of a lattice $\bbZ^d$. 

In the study of Markov semigroups, the spectral gap inequality can be considered in abstract form using the language of generators, carr\'e du champs, and reversible measures; we refer the reader to the works \cite{GuZe03, Bak04, Wan06, BG10, BGL14} for more details. It belongs to the class of so-called coercive functional inequalities and it would be hopeless to give an exhaustive review of all such functional inequalities and their applications. Even within the realm of spectral gap inequalities we may already consider generalisations of \eqref{intro:spectralgap}. For instance, \cite{Wan00} studied a version of \eqref{intro:spectralgap} in which $C > 0$ can be made arbitrarily small at the cost of introducing a possibly arbitrarily large $\bbL^1$-norm, and \cite{RW01} studied the reversed idea of this in which $C$ is made large and a $2$-homogeneous functional of $f$ is made small. We also remark, since it will be of significant interest in the sequel, that \cite{BoZe05} explored a generalisation of \eqref{intro:spectralgap} in which $2$ is replaced with a $q \in (1, 2]$. 

We are interested in spectral gap inequalities for spin systems, which appear naturally in the study of physical phenomena. For instance, one of the earliest models introduced to study ferromagnetism was the Ising model \cite{Isi25}, where the spin space was the product of the two-point space $\{-1, +1\}$ representing the negative and positive states of an elementary magnet respectively. From a mathematical viewpoint, the notion of spin is quite general and indeed in this paper the spin of a particle is regarded as a point in $\bbG$. 
%Thus it may be more accurate to say that the spin space is the space of states a particle may take, and a spin system is essentially a model for the spins of an interacting system of particles. 
The literature on spectral gap inequalities (and also logarithmic Sobolev inequalities) for spin systems is extensive. At least for unbounded spin systems, a good starting point may be the works of \cite{SZ92a, Zeg96, BoHe99, Yos99, GeRo01, Led01} and references within. Our work can be regarded as being at the intersection of works like \cite{Led01, BB05, OtRe07} which studied the general problem of when a finite-dimensional (that is, on a single spin space) spectral gap or logarithmic Sobolev inequality can be promoted to the entire spin space, and the paper \cite{HeZe09} which studied the problem of finding such finite-dimensional inequalities on the Heisenberg group and more generally on nilpotent Lie groups. The Heisenberg group is one of the simplest subelliptic settings for which the rich theory of Bakry-\'Emery calculus \cite{BE85} fails, but \cite{HeZe09} developed a new technology, called a $U$-bound, to prove analogues of known results in the Euclidean setting. (For some recent extensions and applications see \cite{DaZe21a, DaZe21b, DaZe22}, \cite{CFZ20, CFZ22}.)

Our work is partly inspired by \cite{InPa09, InPa14} which combined the two approaches to obtain an infinite-dimensional inequality for the Heisenberg group with interactions dependent on Carnot-Carath\'eodory distance. Here we have a similar objective: we wish to start with the $U$-bounds of \cite{HeZe09} to furnish a finite-dimensional inequality on a single spin space and then analyse the conditions under which passage to the infinite-dimensional inequality is possible. In our endeavour we want to develop a general framework for coercive inequalities for Gibbs measures living on a subspace of an infinite product space which allows the inclusion of more examples than those previously existing in the literature, and in particular go beyond the case of an infinite product space over the Heisenberg group.

In our setup, for a given natural sub-gradient operator whose components generate the entire Lie algebra, one can consider single spin interactions dependent on various metrics. The intricacy of that situation is that dependent on the metric one can have essentially different theory as far as coercive inequalities are concerned. For example, as indicated already in \cite{HeZe09}, if one considers the Carnot-Carath\'eodory distance on the Heisenberg group with a suitable choice of interaction energy then one can prove a logarithmic Sobolev inequality, but if the distance function is smooth outside the origin then the logarithmic Sobolev inequality fails. On the other hand, in the latter case, as demonstrated in \cite{Ing10}, one can do some interesting analysis and prove that the spectral gap inequality persists also for interactions involving the Kaplan norm on the Heisenberg group.

Organisation of the paper is as follows. In Section 2 we introduce spectral gap inequalities and the infinite-dimensional setup. In Section 3 we present a general framework to prove the spectral gap inequalities for Gibbs measures by introducing a weak form of $U$-bound which is suitable for infinite-dimensional theory. In Section 4 we provide a convenient sufficient condition on the interaction to prove the weak $U$-bound. This condition provides us also with a nice way to generalise Dobrushin's compactness to our situation. In Section 5 we examine in detail examples of infinite-dimensional theory on the Heisenberg group, and in particular we provide here examples when the interaction depends on the Kaplan norm. In Section 6 we discuss possible generalisations to other Carnot groups. For more details on nilpotent Lie groups we refer the reader to \cite{BLU07}.

\section{Framework}

\subsection{The spectral gap inequality}
On a smooth probability space $(\Omega, \Sigma, \nu)$ with a smooth sub-gradient $\nabla$, given  $\Gamma^q(f)\equiv |\nabla f|^q$, for $q \in (1, 2]$, one can study a functional inequality of the form
\begin{equation}\label{funceq}
\int_\Omega \abs{f - \int_\Omega f d\nu}^q d\nu \leq C \int_\Omega \Gamma^q(f) d\nu 
\end{equation}
where $C \in (0, \infty)$ does not depend on $f$. Such an inequality is called in the literature the \emph{$q$-spectral gap inequality} or $q$-Poincar\'e inequality (and will be denoted later by $q$-SGI). Its name is derived from the fact that when $q = 2$, the Dirichlet form associated to $\Gamma^2(f) \equiv \abs{\nabla f}^2$ together with the measure $\nu$ uniquely define a self-adjoint operator $L$ on $\bbL^2(\nu)$ which, provided \eqref{funceq} holds, has a spectrum satisfying the containment $\sigma(-L) \sbs \{0\} \cup [1/C, \infty)$, i.e. $-L$ has a spectral gap between $0$ and $1/C$. In particular, if the space is a nilpotent Lie group $\Omega = \bbG$ with Lie algebra $\cG$, there exists a (canonical) collection of $1 \leq K \leq \dim \cG$ vector fields forming the canonical subgradient $\nabla_\bbG = (X^{(1)}, \cdots, X^{(K)})$. With $\Gamma^2(f) = \abs{\nabla_\bbG f}^2$ as the quadratic form and $\Delta_\bbG = \nabla_\bbG \cdot \nabla_\bbG$ as the canonical sub-Laplacian, we can consider the probability measure 
\[ d\nu = \frac{1}{Z}e^{-U}dx \]
defined with a measurable function $U$ such that $0 < Z = \int_\bbG e^{-U}dx < \infty$ and $dx$ denoting the Lebesgue measure on $\bbG$ inherited from the stratification of $\cG$. Then the operator 
\[ L = \Delta_\bbG - \nabla_\bbG U \cdot \nabla_\bbG \]
satisfies 
\[ \inner{f, -Lf}_{\nu} = \int_\bbG \abs{\nabla_\bbG f}^2d\nu, \]
that is $L$ is the self-adjoint operator associated to $\Gamma^2$ and $\nu$. The Euclidean case is entirely analogous and replaces $\nabla_\bbG$ for $\nabla$, and $\Delta_\bbG$ for $\Delta$. 

The $q$-spectral gap inequality for $q \in (1, 2)$ improves the $2$-spectral gap inequality, in the sense if the underlying space is finite-dimensional, the former implies the latter \cite[Proposition~2.3]{BoZe05}. Moreover, the $q$-spectral gap inequality for $q \in (1, 2]$ enjoys two important properties: it is stable under both tensorisation, meaning if $\nu_1$ and $\nu_2$ both satisfy the $q$-spectral gap inequality on $\Omega$, then so too does their tensor product $\nu_1 \otimes \nu_2$ on $\Omega \times \Omega$, and it is also stable with respect to bounded perturbations, in that if $V$ is bounded above and below then $\smash{d\tilde{\nu} = \frac{1}{\tilde{Z}}e^{-U-V}dx}$ satisfies the inequality also.
%[REFERENCE HERE?]

The $q$-spectral gap inequality also plays an important role in the study of the (heat) semigroup $(e^{tL})_{t \geq 0}$ generated by $L$. Namely, it implies exponentially fast convergence $e^{tL}f \rightarrow \int_\Omega fd\nu$ to equilibrium in $\bbL^2(\nu)$. This inequality belongs to a wider class of functional inequalities called the coercive inequalities which derive their namesake from the fact their existence ``coerce'' particular behaviours of semigroups. This class also contains the $q$-logarithmic Sobolev inequalities and the classical Sobolev inequality which, respectively, are related to hypercontractivity and ultracontractivity of the semigroup. We refer the reader to the works \cite{GuZe03, Wan06, BGL14} for more details on coercive inequalities, and to the paper \cite{BoZe05} which first studied the $q$-spectral gap inequality. 
% we already gave some references in intro, maybe not necessary again
\begin{comment}
We refer the reader to \cite[{\S1-2}]{Bak04}, \cite[Chapter~1]{Wan06} and \cite[{\S2-3}]{GuZe03} for more details. 

\end{comment}
%% 

The goal of this paper is to study the spectral gap inequality in the infinite-dimensional setting as well as consider interactions defined in terms of homogeneous norms on $\bbG$ other than the Carnot-Carath\'eodory distance $d$, which was studied in the infinite-dimensional setting in \cite{InPa09}. We will be especially interested in the Kaplan norm $N$ which was proven in \cite[Theorem~4.5.5]{Ing10} to satisfy a finite-dimensional $q$-spectral gap inequality. 
% added reference for InPa09 
\begin{theorem}\label{thm:ingliskaplan}
Let $q \in (1, 2]$. For $N$ the Kaplan norm on the Heisenberg group $\bbH$, $p = q/(q-1)$ the H\"older conjugate to $q$, and $\alpha \in \bbR_{>0}$, the measure
\[ d\nu_p = \frac{1}{Z}e^{-\alpha N^p}dx, \] 
satisfies the $q$-spectral gap inequality. 
\end{theorem}

This inequality is also satisfied by $d$ replacing $N$, and in fact the measure $\frac{1}{Z}e^{-\alpha d^p}dx$ satisfies the stronger $q$-logarithmic Sobolev inequality \cite[Corollary~4.1]{HeZe09}. However as shown in \cite[Theorem~6.3]{HeZe09} a probability defined with any homogeneous norm that is smooth everywhere except at zero, and in particular $N$, can never satisfy the logarithmic Sobolev inequality. Thus  the $q$-spectral gap inequality of Theorem \ref{thm:ingliskaplan} is the next best result that one could hope for. The main point is that although both $d$ and $N$ are both natural candidates as choice of metric, and indeed are topologically equivalent, they nonetheless exhibit very different analytic behaviour from the viewpoint of coercive inequalities. We refer the reader to \cite[Chapter~5]{BLU07} for definitions and properties of $d$, $N$, and homogeneous norms on nilpotent Lie groups in general. 

In the sequel, for $q \in (1, 2]$ a fixed constant we shall reserve the notation $p = q/(q - 1)$ for its H\"older conjugate. Moreover, unless otherwise stated, $\nabla \vcentcolon= \nabla_\bbG$, i.e. we always work with the sub-gradient on $\bbG$ and not the Euclidean gradient. (In general, $\nabla$ can be any nice enough sub-gradient.) In addition, where convenient we adopt the shorthand notation $\nu(f) = \int_\Omega f d\nu$.

\subsection{The infinite-dimensional setting}
We will be interested in functional inequalities of the form \eqref{funceq} where the configuration space $\smash{\Omega = \bbG^{\bbZ^D}}$ is the product of infinitely many copies of a nilpotent Lie group $\bbG$ indexed by the $D$-dimensional integer lattice $\bbZ^D$.   From the point of view of statistical mechanics, each lattice point $i \in \bbZ^D$ represents a particle with unbounded spin taking values in $\bbG$, while a point $\omega \in \Omega$ represents a spin configuration, that is a description of the spin of every particle in the system. 

The measures $\nu$ on $\Omega$ that we consider are constructed in the following way.  
First of all we remark that given a measurable function $f$ on $(\Omega, \Sigma)$ and a subset $\Lambda \sbs \bbZ^D$ we can define a jointly measurable function
$f:\bbG^\Lambda\times\bbG^{\Lambda^c}\to\bbR$ as
\begin{equation}\label{def:conditionalfunction}
f_\Lambda^\omega(x_\Lambda) := f(x_\Lambda \bullet \omega_{\Lambda^c})
\end{equation}
where 
\[ (x_\Lambda \bullet \omega_{\Lambda^c})_j \vcentcolon= \begin{cases}x_j, &\Mfor \, j\in \Lambda\\
\omega_j, &\Mfor \, j\in \Lambda^c.
\end{cases}\]
We define a potential energy  $U_\Lambda^\omega(x_\Lambda)$ for a subsystem in a finite set $\Lambda \Subset \bbZ^D$ given external configuration $\omega \in \Omega$ as follows
\begin{equation}\label{def:potential}
%U_\Lambda^\omega(x_\Lambda) \vcentcolon= \sum_{i \in \Lambda} \phi(x_i) + \sum_{\substack{i, j \in \Lambda \\ j \sim i}} \beta_{ij}V(x_i, x_j) + \sum_{\substack{i \in \Lambda, \, j \notin \Lambda \\ j \sim i}} \beta_{ij}V(x_i, \omega_j).
U_\Lambda^\omega(x_\Lambda) \vcentcolon= U_\Lambda(x_\Lambda \bullet \omega_{\Lambda^c}) = \sum_{i \in \Lambda} \phi(x_i) + \sum_{i, j \in \Lambda} \beta_{ij} V(x_i, x_j) + \sum_{i \in \Lambda, \, j \notin \Lambda} \beta_{ij} V(x_i, \omega_j)
\end{equation}
where $\phi \in C(\bbG, \bbR)$, $V \in C^1(\bbG \times \bbG, \bbR)$, and $\lvert\beta_{ij}\rvert < \beta$ for some $\beta > 0$ and each pair $i, j \in \bbZ^D$. That is, each site $i \in \Lambda$ contributes both a phase part $\phi(x_i)$ and an interaction part $\sum_{j \in \bbZ^D} V(x_i, x_j)$, where $x_j = \omega_j$ whenever $j \notin \Lambda$. 
If necessary we restrict ourselves to a subspace $\mathcal{S}\subset \Omega$ such that
\[ \forall i\in\bbZ^D \Mand \forall \omega\in\mathcal{S}, \quad  \left|\sum_{\substack{i \in \Lambda, \, j \notin \Lambda}} \beta_{ij}V(x_i, \omega_j) \right|<\infty,
\]
% j \sim i removed since we haven't defined \sim, and since 
and 
\[0< Z_\Lambda^\omega \vcentcolon= \int e^{-U(x_\Lambda \bullet \omega_{\Lambda^c})}dx_\Lambda <\infty\]
where $dx_\Lambda$ denotes the Lebesgue measure in $\bbG^\Lambda$. (If the interaction has finite range $R\in(0,\infty)$, meaning $\beta_{ij} \equiv 0$ if $\norm{i - j}\vcentcolon=\norm{i - j}_{l^1(\bbZ^D)} > R$, 
% modified dist(i, j) so as to not need to define dist
then both conditions are satisfied on the entirety of $\Omega$.)
Then we can define a probability measure on $\mathcal{S}\subset \Omega$ by 
\[ d\bbE_\Lambda^\omega  = \delta_{\omega_{\Lambda^c}}\otimes \frac{1}{Z_\Lambda^\omega}e^{-U_\Lambda^\omega}dx_\Lambda \]
where, for given $\omega \in\Omega$, $\delta_{\omega_{\Lambda^c}}$ denotes the point measure on $\bbG^{\Lambda^c}$ concentrated at $\omega_{\Lambda^c}\equiv (\omega_j, j\in\Lambda^c)$. The measure $\bbE_\Lambda^\omega$ restricted to the sigma algebra of Lebesgue measurable subsets of $\bbG^\Lambda$ is called a \emph{local Gibbs measure} (with external condition $\omega\in\mathcal{S}$).
We denote the integral of a function $f: \mathcal{S} \rightarrow \bbR$ against the measure $d\bbE_\Lambda^\omega$ by 
\begin{equation}\label{def:conditionalexpectation}
\bbE_\Lambda^\omega f = \int_{\mathcal{S}} fd\bbE_\Lambda^\omega= \int_{\bbG^\Lambda} \frac{1}{Z_\Lambda^\omega} f(x_\Lambda\bullet\omega_{\Lambda^C}) e^{-U_\Lambda(x_\Lambda\bullet\omega_{\Lambda^c}) }dx_\Lambda
\end{equation}
% modified slightly to introduce later notation $\bbE_\Lambda$
In this way $\bbE_\Lambda^\omega$ can be regarded as a linear functional acting on bounded measurable functions, and it may be promoted to a linear operator $\bbE_\Lambda$ whose action on $f$ is given as $\bbE_\Lambda f(\omega) = \bbE_\Lambda^\omega f$. The family $(\bbE_\Lambda^\omega)_{\Lambda \sbs \bbZ^D}^{\omega\in\mathcal{S}}$ is called a \emph{local specification}.
It preserves the unit function, positivity and satisfies the compatibility condition
\[ \forall \Lambda_1\subset \Lambda_2 \ \forall \omega \in \cS \implies \bbE_{\Lambda_2}^\omega \bbE_{\Lambda_1}^\cdot f = \bbE_{\Lambda_2}^\omega f.\]
A probability measure $\nu$ on $\mathcal{S}$ is called a \emph{global Gibbs measure} for the local specification $(\bbE_\Lambda^\omega)_{\Lambda \Subset \bbZ^D;\,\omega \in \cS}$ if it satisfies the DLR equation 
\[\nu \bbE_\Lambda^\cdot(f) = \nu (f)\] for all $\Lambda \Subset \bbZ^D$ and all integrable functions. This construction, which follows the work of Dobrushin \cite{Dob68a}, Lanford, and Ruelle \cite{LaRu69}, means that the measure $\nu$ is constructed by first providing a description of the behaviour of $\nu$ on finite volume subsystems $\Lambda \Subset \bbZ^D$. The existence of such a measure is not immediately obvious, but Dobrushin gave a criterion \cite[Theorem~1]{Dob70} for spin systems on arbitrary metric spaces which we will use in the sequel to show that the local specifications of interest are well-defined and possess a Gibbs measure $\nu$. Moreover, the proof of the main theorem will go on to show that $\nu$ is unique, which is itself a question of independent interest and related to the physical phenomenon of phase transition \cite{Geo11}. 

We can also now be more precise about the functions $f$ appearing in \eqref{funceq}: they belong to the space $W^{1,q}(\nu)\subset \bbL^q(\nu)$  %\vcentcolon= W^{1,q}(\Omega, \nu)$ %
of functions such that
\[ \nu \abs{\nabla_\Lambda f}^q \vcentcolon= \nu \sum_{i \in \Lambda} \abs{\nabla_if}^q <\infty\]
for all subsets $\Lambda \subset \bbZ^D$, possibly infinite (and in particular possibly all of $\bbZ^D$ itself), and where $\smash{\nabla_if(\omega) \vcentcolon= \nabla_{\{i\}}f_{\{i\}}^\omega(\omega_i)}$ and $\abs{\nabla_if}^q$ is the $q$-th power of the $l^2(\bbZ^D)$-norm of $\nabla_if$. 
\vspace{0.25cm}

In this paper, we concern ourselves with finding conditions on the phase $\phi$ and interaction $V$ so that the Gibbs measure $\nu$ exists, is unique, and satisfies the $q$-SGI \eqref{funceq}. Afterwards we provide applications of our scheme. For that we choose $\bbG$ to be the Heisenberg group and for the sake of notational simplicity, we consider the case of interactions which are of finite range $R = 1$, meaning $\beta_{ij} \equiv 0$ whenever $\norm{i-j}  > 1$.
Already in this setup we obtain a number of generalisations of a few results previously existing in the literature. In particular we provide examples of interactions dependent on a smooth homogeneous norm for which there are no prior results on spectral gap inequalities in the literature for the Gibbs measures of infinite-dimensional systems. The general case of interactions of finite range $1 < R < \infty$ is discussed briefly at the end of the following chapter, and the case of other nilpotent Lie groups is discussed in Section 6. 

% general finite range has been moved from appendix to s3.1 

\begin{comment}
We then discuss the general case of interactions of finite range $1 \leq R < \infty$ in the appendix while in Section 6 we discuss briefly the case of other Carnot groups.
\end{comment}

\section{Passage from single-site inequalities to global inequalities}

In proving inequalities for the global Gibbs measure $\nu$, one may pass from the single-site $q$-SGI to establishing the $q$-SGI holds for every local Gibbs measure $\bbE_\Lambda^\omega$ uniformly in both the finite volume subsystem $\Lambda$ and the boundary conditions $\omega$. In the past this was done, for instance, in \cite[Theorem~4.1]{Zeg96} for the logarithmic Sobolev inequality and in \cite[Theorem~1.1]{GeRo01} for the $2$-SGI on the single spin space $\bbR$. 
% not sure if this part is necessary
\begin{comment}
Here we are interested also in cases when the interaction depends on smooth homogeneous norms on a nilpotent Lie group $\bbG$ and the logarithmic Sobolev inequality cannot hold.
\end{comment}
% maybe add a comment about other implications of uniform inequalities
In the event the $q$-SGI cannot be obtained with constant independent of $\Lambda$ and $\omega$, it may be possible to work with the weaker uniform single-site inequality where $\Lambda = \{i\}$ has unit volume. On the single spin space $\bbH$, the work of \cite[Theorem~5.2]{InPa09} and \cite[Theorem~2.1]{InPa14} establishes the inequality for the Gibbs measure in this way for measures defined using the Carnot-Carath\'eodory distance. 

\begin{comment}
%The meaning of this is is not clear here, so moved to comment.
in which the inequality holds for $\bbE_i^\omega \vcentcolon= \bbE_{\{i\}}^\omega$ with a constant independent of $\omega$. Although it is true that the $q$-SGI is stable under tensorisation, the infinite product of the single-site measures $\bbE_i^\omega$ is in general \emph{not} the Gibbs measure for the given local specification, and so there is some work involved in passing from the uniform single-site inequality to the inequality for Gibbs measure. 

(Y/ I think my original intention was to emphasise the difference between the infinite-dimensional setting and the finite dimensional tensorised setting.) 
\end{comment}

% maybe add more examples of ways to obtain a global SGI 

In this chapter, we show that one can get away with less, that is we show that it suffices that a so-called weak $U$-bound and a weak $q$-SGI holds. $U$-bounds were first studied in the context of nilpotent Lie groups in \cite{HeZe09} where they were used to establish conditions under which a $q$-SGI and other coercive inequalities hold for measures on a finite-dimensional space. The one presented here is weak in the sense the right hand side contains defective terms involving the sub-gradients of $f$ at other sites. It turns out however that this is sufficient, provided that the defective terms decay sufficiently fast as one moves further and further away from the given site $i \in \bbZ^D$. (A strong $U$-bound analogous to \bf{(A1)} would contain only the sub-gradient of $f$ at site $i$, see for instance \cite[Lemma~4.1]{InPa09}.) Similarly, the $q$-SGI here is weak in the sense, in contrast to \eqref{funceq}, the right hand side contains more than just $\nu\abs{\nabla_if}^q$ and instead contains defective terms. 

% maybe this seems a bit unclear
\begin{comment}
Thus our result could be interpreted as saying that we expect each particle to interact with one another provided this interaction decays quickly. 
\end{comment}

\begin{comment}
%Relation to InglisPapageorgiou 
\end{comment}

\begin{theorem}\label{thm:singletoglobal}
%%Section.3 Theorem 3.1%%
Let $q \in (1, 2]$, and let $\nu$ be a Gibbs measure for the local specification $(\bbE_\Lambda^\omega)_{\Lambda \Subset \bbZ^D}^{\omega \in \mathcal{S}}$ given in \eqref{def:potential}. Let $f \in W^{1,q}(\nu)$ and suppose the following two conditions hold for $\beta > 0$ sufficiently small: 
\begin{enumerate}[label=\normalfont{\bf{(A\arabic*)}}]
\item for any $i \in \bbZ^D$ and $j \in \bbZ^D$ with $\norm{j - i}  = 1$, one has the weak single-site $U$-bound
\[ \nu\left(\abs{f}^q\abs{\nabla_jV(x_i, x_j)}^q\right) \leq A\nu\abs{f}^q + A\sum_{k \in \bbZ^D} M_{ki} \nu\abs{\nabla_k f}^q, \]
\item and, for any $i \in \bbZ^D$, one has the weak single-site $q$-SGI 
\[ \nu\bbE_i^\omega \abs{f - \bbE_i^\omega f}^q \leq A \sum_{k \in \bbZ^D} M_{ki} \nu\abs{\nabla_k f}^q \]
\end{enumerate}
for some constant $A > 0$ and some symmetric matrix $(M_{ij}\geq 0)_{i, j \in \bbZ^D}$ such that 
\begin{enumerate}[label=\normalfont{\bf{(B\arabic*)}}]
\item $A = \cO(\beta^{-1/2})$ as $\beta \rightarrow 0^+$, and 
\item $M_{ij} \geq 0$ for all $i, j \in \bbZ^D$ and satisfies $\sup_{i \in \bbZ^D} \sum_{j \in \bbZ^D} M_{ij} < \infty$. 
\end{enumerate}
Then for all $\beta > 0$ sufficiently small the following global $q$-SGI holds
\[ \nu\abs{f - \nu f}^q \leq C_{SG}\nu\abs{\nabla f}^q \]
 with some constant $C_{SG} > 0$ possibly depending on $\beta$,
 for all $f \in W^{1,q}(\nu)$. % satisfying the above conditions. 
\end{theorem}

\begin{remark}
The $l^1(\bbZ^D)$-summability of $(M_{i,j})_{j \in \bbZ^D}$ of \bf{(B2)} implies $\norm{(M_{ij})}_{l^\infty} < \infty$. Without loss of generality we may assume $\norm{(M_{ij})}_{l^\infty} \leq 1$ by rescaling $M$ and absorbing the constant into $A$. 
\end{remark}

The basic idea is that the summability condition \bf{(B2)} implies that if the interaction matrix $M_{ij}$ decays sufficiently fast at infinity (in particular it suffices that it decays exponentially fast), then the defective terms can be dispensed of. The generic case is that $M_{ij}$ is a matrix which decays like $\epsilon^{\norm{i - j}}$ where $0<\epsilon < 1$ but which may be bounded away from zero in a finite neighbourhood of $i$. 

% add more here?
To prove this theorem, we first prove some lemmata. The first is a sweeping out inequality - these were introduced in \cite{GuZe03} in order to prove infinite-dimensional logarithmic Sobolev inequalities. Throughout the remainder of this chapter we fix $q \in (1, 2]$ and assume the conditions of Theorem \ref{thm:singletoglobal} 
%%Section.3 Theorem 2%%
hold. Moreover, we use the notation $j \sim i$ to mean $j$ is a neighbour of site $i$, i.e. $j$ is adjacent to $i$ in the sense $\norm{j - i}  = 1$, and we use the notation $\{\sim i\}$ to denote the set of neighbours of $i$, that is the $2D$ points $j \in \bbZ^D$ with $\norm{i-j}  = 1$. (Thus $j \sim i$ if and only if $j \in \{\sim i\}$.) 
%For convenience, we shall drop the $l^1$-subscript and write $\norm{i} \vcentcolon= \norm{i}_{l^1(\bbZ^D)}$. 
Lastly, whenever we use Landau notation, we implicitly understand it  with respect to the limit $\beta \rightarrow 0^+$; we will need to track the decay (or growth) of some constants and we will typically reserve uppercase letters such as $K$ for large constants that are $\cO(1)$ and lowercase letters such as $c$ for small constants that are $o(1)$. 

\begin{lemma}\label{lem:sweepoutineq}
%Section 3 Lemma 2%%
For all $i, j \in \bbZ^D$, the following single-site sweeping out inequality holds
\begin{equation}\label{sweepoutineq}
\nu\abs{\nabla_j(\bbE_i^\omega f)}^q \leq K_1\nu\abs{\nabla_jf}^q + c_1\sum_{k \in \bbZ^D} M_{ki} \nu\abs{\nabla_kf}^q
\end{equation}
with some constants $K_1 = \cO(1)$ and $c_1 = o(1)$ independent  of $i$ and $f$.
\end{lemma}

\begin{proof}
Let us recall that $\bbE_i^\omega f$ does not depend on $x_i$; in general it depends on every other coordinate $x_j$ for $j \neq i$, but in the specific case of interactions of unit range, it depends on neighbouring spins $x_j$ for $j \sim i$. Thus if $j \notin \{\sim i\}$ then the probability kernel of the measure $d\bbE_i^\omega$ does not depend on $j$, and only $f$ may depend on $j$. It follows that one simply has $\nu \abs{\nabla_j(\bbE_i f)}^q = \nu\abs{\bbE_i(\nabla_jf)}^q \leq \nu\abs{\nabla_jf}^q$ by interchanging differentiation and integration in the first equality and by Jensen's inequality together with the DLR equation in the second inequality, so we shall focus on the case $j \sim i$, i.e. $\norm{i - j} = 1$. Note also that to emphasise the dependence of the function $\bbE_if(\omega) = \bbE_i^\omega f$ on $\omega$, we will write $\bbE_i^\omega f$ with the understanding it replaces $\bbE_if$ even though $\bbE_i^\omega f$ is a real number and $\bbE_if$ is a real-valued function. Let $X^{(l)}$ for $l=1, \cdots, K$ and $K\in\bbN$ be the vector fields generating the Lie algebra $\cG$ of a nilpotent Lie group $\bbG$. For $j\in\bbZ^D$, let $\smash{\nabla_j = (X^{(l)}_j)_{l=1,..,K}}$ denote the sub-gradient on the copy $\bbG^{\{j\}}$ of $\bbG$ at site $j$. 

Thus for a bounded differentiable function with bounded derivative
\[ \nabla_j \bbE_i^\omega f = (X_j^{(l)}(\bbE_i^\omega f) )_{l=1,\cdots,K}. \]
If $\rho_i \vcentcolon= d\bbE_i^\omega / \lambda(dx_i)$ is the density of the measure $\bbE_i^\omega$ with respect to the Lebesgue measure $\lambda(dx_i)$ on $\bbG^{\{i\}}$, then
\begin{equation}\label{gradientcalc1}
X_j^{(l)}(\bbE_i^\omega f) = \int \left(\rho_i(X_j^{(l)}f) + (X_j^{(l)}\rho_i)f\right) \lambda(dx_i) = \bbE_i^\omega(X_j^{(l)}f) + \int (X_j^{(l)}\rho_i)f \lambda(dx_i).
\end{equation}
A direct calculation yields
\begin{align*}
X_j^{(l)}\rho_i = X_j^{(l)} \left(\frac{e^{-U_i^\omega}}{\int e^{-U_i^\omega} \lambda(dx_i)}\right) &= \frac{(-X_j^{(l)}U_i^\omega) e^{-U_i^\omega}}{\int e^{-U_i^\omega} \lambda(dx_i)} - \frac{e^{-U_i^\omega} \int (-X_j^{(l)}U_i^\omega) e^{-U_i^\omega} \lambda(dx_i)}{(\int e^{-U_i^\omega} \lambda(dx_i))^2} \\
&= -\rho_i(X_j^{(l)}U_i^\omega) + \rho_i\bbE_i^\omega(X_j^{(l)}U_i^\omega)
\vphantom{\frac{(-X_j^{(l)}U_i^\omega) e^{-U_i^\omega}}{\int e^{-U_i^\omega} \lambda(dx_i)} - \frac{e^{-U_i^\omega} \int (-X_j^{(l)}U_i^\omega) e^{-U_i^\omega} \lambda(dx_i)}{(\int e^{-U_i^\omega} \lambda(dx_i))^2}}
\end{align*}
and therefore 
% 16-11-22 changed (k) to (l)
\begin{equation}\label{gradientcalc2}
\begin{aligned}
\int(X_j^{(l)}\rho_i)f\lambda(dx_i) &= -\bbE_i^\omega((X_j^{(l)}U_i^\omega)f) + (\bbE_i^\omega f)(\bbE_i^\omega(X_j^{(l)}U_i^\omega)) \\
&= -\bbE_i^\omega((f - \bbE_i^\omega f)(X_j^{(k)}U_i^\omega)) 
\vphantom{\int(X_j^{(l)}\rho_i)f\lambda(dx_i)}
\end{aligned}
\end{equation}
Hence we have
\begin{align*}
\nabla_j(\bbE_i^\omega f) &= (\bbE_i^\omega  \nabla_j f) 
-\bbE_i^\omega((f - \bbE_i^\omega f)(\nabla_j U_i^\omega)) \\
&= (\bbE_i^\omega  \nabla_j f) 
-\beta_{ij} \bbE_i^\omega((f - \bbE_i^\omega f)( \nabla_j V(x_i, \omega_j)))
\end{align*}
Inserting \eqref{gradientcalc2} into \eqref{gradientcalc1}, we obtain 
\begin{align*}
\abs{\nabla_j \bbE_i^\omega f}^q &\leq 2(\bbE_i^\omega \abs{\nabla_j f}^q + \bbE_i^\omega(\abs{f - \bbE_i^\omega f}^q \abs{\nabla_j U_i^\omega}^q) \\
&\leq 2(\bbE_i^\omega \abs{\nabla_jf}^q + \beta^q \bbE_i^\omega(\abs{f - \bbE_i^\omega f}^q \abs{\nabla_j V(x_i, \omega_j)}^q)
\end{align*}
%by Jensen's inequality to pull out the $\bbE_i^\omega$ and the fact that $\smallabs{X_j^1f}^q + \smallabs{X_j^2f}^q \leq \abs{\nabla_jf}^q$. If we now integrate both sides with respect to $\nu$, we find
by Jensen's inequality. Integrating with respect to $\nu$, we find
\begin{equation}\label{gradEifbound}
\begin{aligned}
\nu\abs{\nabla_j \bbE_i^\omega f}^q &\leq 2\nu\abs{\nabla_jf}^q + 2\beta^q\nu(\abs{f - \bbE_i^\omega f}^q \abs{\nabla_j V(x_i, x_j)}^q 
%%%%%%%
\vphantom{\left(\nu\abs{f - \bbE_i^\omega f}^q + \sum_{k \in \bbZ^D} C_\beta^{\norm{k - i}}\nu\abs{\nabla_k(f - \bbE_i^\omega f)}^q\right)} \\
%%%%%%%
&\leq 2\nu\abs{\nabla_j f}^q + 2A\beta^q\left(\nu\abs{f - \bbE_i^\omega f}^q + \sum_{k \in \bbZ^D} M_{ki} \nu\abs{\nabla_k(f - \bbE_i^\omega f)}^q\right)
\end{aligned}
\end{equation}
where the second inequality is by \bf{(A1)}. For the first term in the bracket, it holds
\begin{equation}\label{gradEifbound1}
\nu\abs{f - \bbE_i^\omega f}^q \leq A \sum_{k \in \bbZ^D} M_{ki} \nu\abs{\nabla_k f}^q 
\end{equation}
by \bf{(A2)}, and for the second term, it holds
\begin{equation}\label{gradEifbound2}
\begin{aligned}
\sum_{k \in \bbZ^D} &M_{ki}\nu\abs{\nabla_k(f - \bbE_i^\omega f)}^q \\
&\leq 2\sum_{k \in \bbZ^D} M_{ki} \left(\nu\abs{\nabla_kf}^q + \nu\abs{\nabla_k\bbE_i^\omega f}^q\right) \\
&\leq 4\sum_{k \in \bbZ^D} M_{ki} \nu\abs{\nabla_kf}^q + 2\sum_{j \neq k \sim i} M_{ki}\nu\abs{\nabla_k\bbE_i^\omega f}^q + 2\nu\abs{\nabla_j\bbE_i^\omega f}^q
\end{aligned}
\end{equation}
where in the third inequality we have simplified $\nu\abs{\nabla_k \bbE_i^\omega f}^q$ by considering the two cases $k \sim i$ and $k \not\sim i$ separately. In the first case, $\norm{(M_{ij})}_{L^\infty} \leq 1$ by \bf{(B2)} so
\begin{align*}
2\sum_{k \sim i} M_{ki} \nu\abs{\nabla_k\bbE_i^\omega f}^q &= 2\sum_{j \neq k \sim i} M_{ki} \nu\abs{\nabla_k\bbE_i^\omega f}^q + 2M_{ji} \nu\abs{\nabla_j\bbE_i^\omega f}^q \\
&\leq 2\sum_{j \neq k \sim i} M_{ki} \nu\abs{\nabla_k\bbE_i^\omega f}^q + 2\nu\abs{\nabla_j\bbE_i^\omega f}^q,
%%%%%%%
\vphantom{2\sum_{\substack{k \sim i \\ k \neq j}} (M_{ki} + M_{kj}) \nu\abs{\nabla_k\bbE_i^\omega f}^q + 2(M_{ji} + M_{jj}) \nu\abs{\nabla_j\bbE_i^\omega f}^q}
%%%%%%%
\end{align*}
and in the second case $\nu\abs{\nabla_k \bbE_i^\omega f}^q \leq \nu\abs{\nabla_k f}^q$ and thus it can be absorbed into the first sum $2\sum_{k \in \bbZ^D} M_{ki}\nu\abs{\nabla_kf}^q$. Inserting \eqref{gradEifbound1} and \eqref{gradEifbound2} back into \eqref{gradEifbound}, we find
\begin{equation}\label{gradEifboundInduction}
\nu\abs{\nabla_j\bbE_i^\omega f}^q \leq a_1\nu\abs{\nabla_j f}^q + a_2\sum_{k \in \bbZ^D} M_{ki}\nu\abs{\nabla_k f}^q + a_3\sum_{j \neq k \sim i} M_{ki} \nu\abs{\nabla_k\bbE_i^\omega f}^q 
\end{equation}
with 
\[ a_1 \vcentcolon= \frac{2}{1 - 4A\beta^q}, \quad a_2 \vcentcolon= \frac{2A\beta^q(A + 4)}{1 - 4A\beta^q}, \quad a_3 \vcentcolon= \frac{4A\beta^q}{1 - 4A\beta^q}. \]
By the estimate on $A$ of \bf{(B1)}, we may assume $a_1 < 4$ and $a_3 < a_2 < \beta^{(q-1)/2} < 1$, for $\beta$ small enough. The difference between \eqref{gradEifboundInduction} with \eqref{sweepoutineq} is the appearance of the summation $a_3 \sum_{j \neq k \sim i} M_{ki} \nu\abs{\nabla_k\bbE_i^\omega f}^q$ which we claim can be eliminated by the following induction procedure.

For notational convenience, let us denote by
\[ W_k \vcentcolon= \nu\abs{\nabla_k \bbE_i^\omega f}^q, \quad Y_k \vcentcolon= \nu\abs{\nabla_kf}^q, \quad S \vcentcolon= \sum_{k \in \bbZ^D} M_{ki}Y_k, \]
in which case \eqref{gradEifboundInduction} reads, for $\beta>0$ small enough and $a_2 > a_3$ replacing $a_3$ for simplicity,
\begin{equation}\label{gradEifboundInductionXYZ}
\begin{aligned}
W_j &\leq a_1Y_j + a_2S + a_2\sum_{j \neq k \sim i} M_{ki} W_k.
%&\leq a_1Y_j + a_2Z + a_2\sum_{j \neq k \sim i} (M_{ki} + M_{kj}) W_k. 
\end{aligned}
\end{equation}
We can apply \eqref{gradEifboundInductionXYZ} on each $W_k$ appearing in the sum; one iteration yields 
\begin{align*}
a_2\sum_{j \neq k \sim i} M_{ki} W_k &\leq a_2\sum_{j \neq k \sim i} M_{ki} \left(a_1Y_k + a_2S + a_2\sum_{k \neq \ell \sim i} M_{\ell i} W_\ell\right) \\
&\leq a_1a_2S + 2da_2^2S + 2da_2^2\sum_{k \neq \ell \sim i} M_{\ell i} W_\ell 
%%%%%%%
\vphantom{a_2\sum_{k \sim i} \delta\left(a_1Y_k + a_2S + a_2\sum_{j \neq \ell \sim i} \delta W_\ell\right)} \\
%%%%%%%
&\leq a_1a_2S + a_2(2da_2)S + a_2(2da_2)\sum_{k \neq \ell \sim i} M_{\ell i} W_\ell
%%%%%%%
\vphantom{a_2\sum_{k \sim i} \delta\left(a_1Y_k + a_2S + a_2\sum_{j \neq \ell \sim i} \delta W_\ell\right)}
%%%%%%%
\end{align*}
since $\sum_{j \neq k \sim i} M_{\ell i} Y_k \leq S$ and $\norm{(M_{ij})}_{l^\infty} \leq 1$ by \bf{(B2)}, and each site has exactly $2D$ neighbours. With $\epsilon = 2Da_2 = o(1)$, we obtain after a second iteration
\begin{align*}
W_j &\leq a_1Y_j + a_2S + a_1a_2S + \epsilon a_2S + \epsilon a_2 \sum_{j \neq \ell \sim i} M_{\ell i} W_\ell \\
&\leq a_1Y_j + a_2S + a_1a_2S + \epsilon a_2S + \epsilon a_1a_2S + \epsilon^2 a_2S + \epsilon^2 a_2\sum_{j \neq m \sim i} M_{mi} W_m
\end{align*}
and so on and so forth; iterating this procedure gives in the limit 
\[ W_j \leq a_1Y_j + \frac{a_2 + a_1a_2}{1 - \epsilon}S \]
by taking $\epsilon = \cO(a_2) = o(1)$ small enough. In particular if $\beta$ is small enough then since $a_1 < 4 = \cO(1)$ and $a_2 < \beta^{(q-1)/2}$ one obtains the desired inequality with $K_1 = 4$ and $c_1 = \beta^{(q-1)/4}$. 
\end{proof}

\begin{remark}
The estimates given here are somewhat wasteful but can easily be improved. For instance $a_2 = \cO(\beta^{q-1})$ so $c_1$ can be taken as $\cO(\beta^{q-1})$ as well, but the choice of estimate is intentional and eliminates the need to write down the controlling constant. 
\end{remark}

\begin{remark}
For the limit of the inductive procedure to be well-defined, we need to show $\sum_{j \neq k \sim i} M_{ki} W_k \leq \sum_{k \sim i} M_{ki} W_k = \nu\abs{\nabla_{\{\sim i\}} \bbE_i^\omega f}^q \leq \nu\abs{\nabla f}^q$ is bounded. Adding $a_2M_{ji}W_j$ to the right hand side of \eqref{gradEifboundInductionXYZ} for simplicity and summing over $j \sim i$ gives
\begin{align*}
\sum_{j \sim i} W_j &\leq a_1 \sum_{j \sim i} Y_j + 2da_2S  + 2da_2\sum_{k \sim i} M_{ki} W_k \\
&\leq  (a_1 + 2da_2) \nu\abs{\nabla f}^q + 2da_2\sum_{k \sim i} W_k. 
\end{align*}
Taking the sum $2Da_2\sum_{k \sim i} W_k$ to the left hand side we obtain $(1 - 2da_2)\sum_{j \sim i} W_j \leq (a_1 + 2da_2)\nu\abs{\nabla f}^q$; since $a_2 = o(1)$, if $\beta$ is small enough then 
\[\sum_{j \sim i} W_j \leq 4a_1\nu\abs{\nabla f}^q < 16\nu\abs{\nabla f}^q < \infty\] 
by assumption of $f \in W^{1,q}(\nu)$. 
\end{remark}

Using the single-site sweeping out inequality \eqref{sweepoutineq}, we will prove a global sweeping out inequality on a carefully chosen infinite volume subsystem. The first step is to form a finite partition of $\bbZ^D$ whose components are infinite volume subsystems $(\Gamma_n)_{n=0}^{N-1}$ with the property that no interaction exists within each $\Gamma_n$ which, in our case, translates to each pair of points $i, j \in \Gamma_n$ satisfying $\norm{i - j} > 1$ whenever $i \neq j$. The partition chosen here will be consist of $N = 2^D$ translations of the cube $\Gamma_0 = (2\bbZ)^D$ translating $0$ to another point in $[0, 1]^D \sbs \bbZ^D$, although we remark that this is not the only partition. Indeed, one can also use the partition $\Gamma_0 = \{i \in \bbZ^D \mid \norm{i} \in 2\bbZ\}$ with $\Gamma_1 = \bbZ^D \setminus \Gamma_0$ as in \cite[\S5]{InPa09}; the choice of partition enters the proof through the controlling constants so the partitions are equivalent as far as obtaining the $q$-SGI (with no regard for the constant) is concerned. 

We will give a sketch of the proof of how to generalise the results to interactions of finite range $R > 1$ later, and it shall be convenient to adopt the former partition here instead. (The latter partition generalises to finite range $R > 1$, but the proof for the former is somewhat easier.) In any case we fix in the sequel our choice of partition $(\Gamma_n)_{n=0}^{N-1}$ as above, and we begin by proving global sweeping out inequalities for $\Gamma_n$; the self-similarity of the $\Gamma_n$ will mean the proof follows for all if it follows for one. 

% I think this was a mistake? no dependence on i here? change n -> N
\begin{lemma}\label{lem:sweepoutGamma}
For any $\Lambda \sbs \bbZ^D$ and $n \in \{0, 1, \cdots, N - 1\}$, the sub-lattice sweeping out inequality
\begin{equation}\label{sweepoutGamma}
\nu\abs{\nabla_\Lambda (\bbE_{\Gamma_n}f)}^q \leq K_2\nu\abs{\nabla_\Lambda f}^q + c_2\nu\abs{\nabla_{\Lambda^c}f}^q
\end{equation}
holds for some constants $K_2 = \cO(1)$ and $c_2 = o(1)$ independently of $n$ and $\Lambda$. 
\end{lemma}

\begin{remark}
We emphasise the reason that the sub-gradient is taken over a general subset $\Lambda \sbs \bbZ^D$ of the lattice here instead of at a single site $j \in \bbZ^D$ as in Lemma \ref{lem:sweepoutineq} is because in the latter case the case $j \not\sim i$ is trivial and thus the sweeping out inequality is only meaningful for $\Lambda = \{\sim i\}$. 
\end{remark}

\begin{proof}
As mentioned earlier, it suffices to prove the case $n = 0$. First, rewrite the gradient into its single-site form 
\begin{equation}\label{SGIsweepoutGamma}
\nu\abs{\nabla_\Lambda(\bbE_{\Gamma_0} f)}^q = \sum_{j \in \Lambda} \nu\abs{\nabla_j(\bbE_{\Gamma_0}f)}^q.
\end{equation}
Let $i \in \Gamma_0$ and $j \in \Lambda$. Since $\Gamma_0$ has no self-interaction, the measure $\bbE_{\Gamma_0}$ can be realised as the product of the single-site measures $\bbE_{\Gamma_0} = \bbE_{i_1}\bbE_{i_2} \cdots$ where $(i_m)_{m \geq 1}$ enumerates $\Gamma_0$. This reduces the problem of studying $\nabla_j\bbE_{\Gamma_0}$ to the problem of studying $\nabla_j\bbE_i$ for $i \in \Gamma_0$. 

Note if $\abs{i_m - j} > 1$, that is if site $j$ does not interact with site $i_m$, then $\nabla_j \bbE_{i_m}^\omega = \bbE_{i_m}^\omega \nabla_j f$ together with Jensen's inequality allows us to pull $\bbE_{i_m}^\omega$ out, which is then eliminated by $\nu$ with the DLR equation, so the only issue are sites $i_m \in \Gamma_0$ such that $i_m \sim j$. Thinking of $\Gamma_0$ as partitioning $\bbZ^D$ into $D$-dimensional cubes of edge length $2$, each containing exactly $N = 2^D$ points in $\Gamma_0$, it follows there is at least one cube containing $j$. (It might be the case $j$ falls on the boundary of two cubes; we may without loss of generality choose the cube whose centre is closest to the origin.) Denote by $Q_j$ this cube. By construction, 
\[ \Gamma_0 \cap \{\sim j\} \sbs \Gamma_0 \cap Q_j \vcentcolon= \{i_1^j, i_2^j, \cdots, i_N^j\} \sbs \Gamma_0 \]
where $(i_m^j)_{m = 1, 2, \cdots, N}$ enumerates $\Gamma_0 \cap Q_j$. Thus we can reduce the sites in \eqref{SGIsweepoutGamma} to
\begin{equation}\label{SGIsweepoutGamma0}
\nu\abs{\nabla_\Lambda(\bbE_{\Gamma_0}f)}^q \leq \sum_{j \in \Lambda} \nu\abs{\nabla_j(\bbE_{Q_j}f)}^q = \sum_{j \in \Lambda} \nu\abs{\nabla_j(\bbE_{i_N^j} \cdots \bbE_{i_2^j} \bbE_{i_1^j}f)}^q
\end{equation}
Let us fix $j$ and drop the superscript, that is write $i_m$ for $i_m^j$. Denote by $f_0 \vcentcolon= f$ and $f_m \vcentcolon= \bbE_{i_m} \cdots \bbE_{i_2} \bbE_{i_1} f$ for $m = 1, 2, \cdots, N$, so that $\nu\abs{\nabla_j(\bbE_{\Gamma_0}f)}^q \leq \nu\abs{\nabla_jf_N}^q$. By Lemma \ref{lem:sweepoutineq}, we have 
\begin{equation}\label{f2dInduction}
\nu\abs{\nabla_j f_m}^q \leq K_1\nu\abs{\nabla_j f_{m-1}}^q + c_1 \sum_{k \in \bbZ^D} M_{k, i_m}\nu\abs{\nabla_k f_{m-1}}^q.
\end{equation}
Note that if $k$ is sufficiently far from $Q_j$ then site $k$ does not interact with any of the $N$ sites in $Q_j$. We can make this observation more precise: if $R_j$ is the smallest cube strictly containing $Q_j$, then for $k \notin R_j$ one has $\nu\abs{\nabla_k f_m}^q \leq \nu\abs{\nabla_k f}^q$ by Jensen's inequality, and 
\[ M_{k, i_m} \leq M_{k, Q_j} \vcentcolon= \sup_{i \in Q_j} M_{k, i} \]
since $i_m \in Q_j$. Moreover, $M_{k, i_m} \leq 1$ by \bf{(B2)}, so 
\begin{equation}\label{f2dInductionStep}
\nu\abs{\nabla_j f_{m}}^q \leq K_1\nu\abs{\nabla_j f_{m - 1}}^q + c_1\sum_{k \in R_j}\nu\abs{\nabla_k f_{m - 1}}^q + c_1\sum_{k \notin R_j} M_{k, Q_j} \nu\abs{\nabla_kf}^q.
\end{equation}

In fact, the only dependence on $j$ in this equation comes from the fact a choice of $j$ gives a choice of $Q_j$ and therefore a choice of sites $i_m$ for $m = 1, 2, \cdots, N$ with respect to which we take a conditional expectation. This means we may replace $\nabla_j$ with $\nabla_k$ for $k \in \bbZ^D$ without replacing $Q_j$, and thus $(i_m) = (i_m^j)$, with $Q_k$ and $(i_m^k)$ respectively. We can then apply Lemma \ref{lem:sweepoutineq} as before, and, again, analyse the summation through the decomposition $\bbZ^D = R_j \sqcup (R_j)^c$ and not $\bbZ^D = R_k \sqcup (R_k)^c$. Thus we arrive at
\begin{equation}\label{f2dInductionStepk}
\nu\abs{\nabla_k f_m}^q \leq K_1\nu\abs{\nabla_k f_{m-1}}^q + c_1\sum_{\ell \in R_j}\nu\abs{\nabla_\ell f_{m-1}}^q + c_1\sum_{\ell \notin R_j} M_{\ell, Q_j} \nu\abs{\nabla_\ell f}^q.
\end{equation}
In fact we only need this for $k \in R_j$ but the statement holds nonetheless unconditionally.

%The appearance of the $K$ is somewhat notationally cumbersome so if $K \leq 1$, then we will control $D_2K$ by $D_2$; otherwise, if $K > 1$, we will redefine $D_2 \vcentcolon= D_2K$. For the purposes of the proof this poses no problem as $D_2 \rightarrow 0$ as $\beta \rightarrow 0$ so $D_2K$ can be always made small enough. 

We now follow an inductive procedure in the spirit of the proof of Lemma \ref{lem:sweepoutineq}. Denote by
\[ W_m^j \vcentcolon= \nu\abs{\nabla_j f_m}^q, \quad Y_m^j \vcentcolon= \sum_{k \in R_j} \nu\abs{\nabla_k f_m}^q = \sum_{k \in R_j} W_m^k, \quad S^{(j)} \vcentcolon= \sum_{k \notin R_j} M_{k, Q_j}\nu\abs{\nabla_k f}^q \]
so that \eqref{f2dInductionStepk} now reads
\begin{equation}\label{f2dInductionWYZ}
W_m^k \leq K_1W_{m-1}^k + c_1Y_{m-1}^j + c_1S^{(j)}.
\end{equation}
Noting $\abs{R_j} = 5^D$, we can sum the left hand side over $k \in R_j$ to conclude
\begin{equation}\label{f2dInductionWYZ2}
\begin{aligned}
Y_m^j \leq \sum_{k \in R_j} (K_1W_{m-1}^k + c_1Y_{m-1}^j + c_1S^{(j)}) &\leq (K_1 + 5^Dc_1)Y_{m-1}^j + 5^Dc_1S^{(j)} \\
%&\leq 5^D(1 + c_1)Y_{m-1}^j + 5^Dc_1Z \vphantom{\sum_{k \in R_j} (K_1W_{m-1}^k + c_1Y_{m-1}^j + c_1Z^j)} \\
%&\leq 6^dY_{m-1}^j + 6^dZ. \vphantom{\sum_{k \in R_j} (K_1W_{m-1}^k + c_1Y_{m-1}^j + c_1Z^j)}
&\leq 5Y_{m-1}^j + S^{(j)} \vphantom{\sum_{k \in R_j} (K_1W_{m-1}^k + c_1Y_{m-1}^j + c_1S^{(j)})} \\  
\end{aligned}
\end{equation}
for $\beta$ small enough, since $K_1 = 4$ and $c_1 = o(1)$. Note if we iterate \eqref{f2dInductionWYZ} into itself, we find 
\begin{align*}
W_{m-1}^j &\leq K_1W_{m-2}^j + c_1Y_{m-2}^j + c_1 S^{(j)} 
%%%%%%%
\vphantom{K_1^{m-1}W_0^j +  c_1\sum_{i=2}^m K_1^{i-2} Y_{m-i}^j + c_1\sum_{i=2}^m K_1^{i-2} S^{(j)}} \\
%%%%%%%
&\leq c_1Y_{m-2}^j + c_1S^{(j)} + K_1^2W_{m-3}^j + K_1c_1Y_{m-3}^j + K_1c_1S^{(j)} 
%%%%%%%
\vphantom{K_1^{m-1}W_0^j +  c_1\sum_{i=2}^m K_1^{i-2} Y_{m-i}^j + c_1\sum_{i=2}^m K_1^{i-2} S^{(j)}} \\
%%%%%%%
&\mathrel{\makebox[\widthof{=}]{\vdots}}
%%%%%%%
\vphantom{K_1^{m-1}W_0^j +  c_1\sum_{i=2}^m K_1^{i-2} Y_{m-i}^j + c_1\sum_{i=2}^m K_1^{i-2} S^{(j)}} \\
%%%%%%%
&\leq K_1^{m-1}W_0^j +  c_1\sum_{i=2}^m K_1^{i-2} Y_{m-i}^j + c_1\sum_{i=2}^m K_1^{i-2} S^{(j)}
\end{align*}
so that $K_1W_{m-1}^j \leq K_1^mW_0^j + c_1\sum_{i=1}^m K_1^i Y_{m-1-i}^j + c_1\sum_{i=1}^m K_1^i S^{(j)}$. Similarly, iterating \eqref{f2dInductionWYZ2} into itself gives 
\[ Y_m^j \leq 5^m (Y_0^j + S^{(j)}) \leq 5^N (Y_0^j + S^{(j)}) \]
so that
\begin{equation}\label{f2dInductionWYZ3}
\begin{aligned}
\nu\abs{\nabla_j f_{N}}^q = W_{N}^j &\leq K_1^{N}W_0^j + c_1\sum_{i=0}^{N} K_1^iY_{N-1-i}^j + c_1\sum_{i=0}^{N} K_1^i S^{(j)}
%%%%%%%
\vphantom{K_1^{N}W_0^j + c_1\sum_{i=0}^{N} K_1^iY_{N-1-i}^j + c_1\sum_{i=0}^{N} K_1^i S^{(j)}} \\
%%%%%%%
&\leq L_1W_0^j + c_1L_2Y_0 + c_1L_2 S^{(j)}
\end{aligned}
\end{equation}
for $L_1 = K_1^N$ and $L_2 = (N + 1)(5K_1)^N$. If we take this back into the summation \eqref{SGIsweepoutGamma0}, we find 
\[ \nu\abs{\nabla_\Lambda(\bbE_{\Gamma_0}f)}^q \leq \sum_{j\in\Lambda} \left(L_1\nu\abs{\nabla_jf}^q + c_1L_2\sum_{k \in R_j}\nu\abs{\nabla_kf}^q + c_1L_2\sum_{k \notin R_j} M_{k, Q_j} \nu\abs{\nabla_kf}^q\right). \]
Let us denote by $I_1, I_2, I_3$ respectively the three terms. 

%\begin{remark}
%In general the proof can be performed with a different choice of sub-lattices; in the case of range $R = 1$ interactions, a more straightforward choice is $\Gamma_0 = \{i \in \bbZ^D \mid \norm{i} \in 2\bbZ\}$ and $\Gamma_0 = (\Gamma_0)^c$. These choices enter the proof through the estimate on $a_3$, i.e. the choices are equivalent as far as obtaining the $q$-SGI is concerned modulo the controlling constant. 
%\end{remark}

$I_1$, by definition, is $L_1 \nu\abs{\nabla_\Lambda f}^q$, while estimating $I_2$ boils down to a counting exercise: if $O_k = \{j \in \bbZ^D \mid k \in R_j\}$ then $\abs{S_k}$ is finite and bounded above uniformly by say $L_3$. Indeed, if $j$ is sufficiently far from $k$ then $R_j$ is far from $k$ also, that is there is a finite neighbourhood of $k$ which contains $O_k$. It follows
\begin{equation}\label{sweepoutI2}
I_2 = c_1L_2\sum_{j \in \Lambda}\sum_{k \in R_j} \nu\abs{\nabla_k f}^q \leq c_1L_2\sum_{k \in \bbZ^D}\sum_{j \in O_k} \nu\abs{\nabla_k f}^q \leq c_1L_2L_3\nu\abs{\nabla f}^q.
\end{equation}

To estimate $I_3$, we expand the summation from $j \in \Lambda$ and $k \in R_j$ to $j, k \in \bbZ^D$, and following the arguments for $I_2$, we conclude
\[ \sum_{j \in \bbZ^D} M_{k, Q_j} \leq \sum_{j \in \bbZ^D} \sum_{\ell \in Q_j} M_{k, \ell} \leq \sum_{j \in \bbZ^D} \sum_{\ell \in R_j} M_{k, \ell} = \sum_{\ell \in \bbZ^D} \sum_{j \in O_\ell} M_{k, \ell} \leq L_3\sum_{\ell \in \bbZ^D} M_{k, \ell} \leq L_3L_0 \eqqcolon L_4 \]
for $L_0 = \sup_{i \in \bbZ^D} \sum_{j \in \bbZ^D} M_{ij}$ the supremum of \bf{(B2)}. It follows
\begin{equation}\label{sweepoutI3}
\begin{aligned}
I_3 &\leq c_1L_2\sum_{j \in \bbZ^D} \sum_{k \in \bbZ^D} M_{k, Q_j}\nu \abs{\nabla_kf}^q \\
&\leq c_1L_2 \sum_{k \in \bbZ^D} \nu \abs{\nabla_k f}^q \sum_{j \in \bbZ^D} M_{k, Q_j} \leq c_1L_2L_4 \nu\abs{\nabla f}^q.
\end{aligned}
\end{equation}

Since $\abs{\nabla f}^q = \smallabs{\nabla_\Lambda f}^q + \smallabs{\nabla_{\Lambda^c} f}^q$, by collecting the estimates $I_1 = L_1\nu\abs{\nabla_\Lambda f}^q$ and \eqref{sweepoutI2}, \eqref{sweepoutI3}, this proves the lemma with $K_2 = L_1 + c_1L_2(L_3 + L_4)$ and $c_2 = c_1L_2(L_3 + L_4)$. More concretely for $\beta$ small enough we may take $K_2 = 4L_1 = 4^{N+1}$ and $c_2 = \beta^{(q-1)/8}$ since $c_2 = \cO(c_1) = \cO(\beta^{(q-1)/4})$. 
\end{proof}

The next problem is to use the sweeping out inequality on the sub-lattice to transfer the weak single-site $q$-SGI to a weak sub-lattice $q$-SGI; in particular we want to show the weak $q$-SGI holds for each of the $N$ sub-lattices $\Gamma_n$.  

\begin{lemma}\label{lem:weakSGIGamma}
For any $n \in \{0, 1, \cdots, N-1\}$, the following weak sub-lattice $q$-SGI 
\begin{equation}\label{weakSGIGamma}
\nu\bbE_{\Gamma_n}^\omega\abs{f - \bbE_{\Gamma_n}^\omega f}^q \leq C_0\nu\abs{\nabla f}^q
\end{equation}
holds for some constant $C_0 > 0$ independently of $n$ such that $C_0 = \cO(\beta^{-1/2})$. 
\end{lemma}

Note that this is still a weak $q$-SGI as the right hand side contains $\abs{\nabla f}^q$ instead of $\abs{\nabla_{\Gamma_n}f}^q$; but unlike the weak $q$-SGI of \bf{(A2)} the interaction matrix $M_{ki}$ is conspicuously absent. In a sense this reflects the fact the sub-lattices infiltrate the lattice so that the distance between an arbitrary point and the sub-lattice is bounded above. 

\begin{proof}
We again consider only the case $i = 0$, the other cases being similar. Let $(i_m)_{m \geq 1}$ enumerate $\Gamma_0$, and denote by as before $f_0 \vcentcolon= f$ and $f_m = \bbE_{i_m} \cdots \bbE_{i_2}\bbE_{i_1}f$, so that $\bbE_{\Gamma_0}f = f_\infty$. By the DLR equation and a simple telescoping argument it can be shown that the following interpolation identity holds: 
\[ \nu\bbE_{\Gamma_0}\abs{f - \bbE_{\Gamma_0}f}^q = \sum_{m=0}^\infty \nu\left(\bbE_{i_{m+1}}^\omega\abs{f_m - \bbE_{\Gamma_0}f}^q - \abs{\bbE_{i_{m+1}}^\omega(f_m - \bbE_{\Gamma_0}f)}^q\right). \]
(In particular the negative part of the summand at $m$ cancels with the positive part of the summand at $m+1$ so that as $m \rightarrow \infty$ only the $\nu\bbE_{i_1}^\omega\abs{f_0 - \bbE_{\Gamma_0}f}^q = \nu\bbE_{\Gamma_0}\abs{f - \bbE_{\Gamma_0}f}^q$ part survives.) To estimate the difference, we use the following sub-quadratic inequality
\[ \abs{a}^q - \abs{b}^q \leq \text{sgn}(b) q \abs{b}^{q-1}(a - b) + 2\abs{a - b}^q \]
which holds for all $a, b \in \bbR$ and $q \in (1, 2]$. We can take $a = f$ and $b = \bbE_i^\omega f$, and integrate with respect to $\bbE_i^\omega$ and obtain 
\begin{equation}\label{JensenGapVarianceIneq}
\bbE_i^\omega\abs{f}^q - \abs{\bbE_i^\omega f}^q \leq 2\bbE_i^\omega\abs{f - \bbE_i^\omega f}^q.
\end{equation}
(The first term vanishes since $\bbE_i^\omega f$ is a constant, and so $\bbE_i^\omega(f - \bbE_i^\omega f) = 0$.) Using this bound, we obtain 
\begin{equation}\label{weakSGIGamma0}
\nu\abs{f - \bbE_{\Gamma_0}f}^q \leq 2\sum_{m=0}^\infty \nu\bbE_{i_{m+1}}^\omega\abs{f_m - \bbE_{i_{m+1}}^\omega f_m}^q.
\end{equation}
Applying \bf{(A2)} to each summand, we find 
%\begin{align}
%\nu\bbE_{i_{m+1}}^\omega\abs{f_m - \bbE_{i_{m+1}}^\omega f_m}^q &\leq B_{SG}\sum_{j \in \bbZ^D} N_{j, i_{m+1}} \nu\abs{\nabla_j f_m}^q \nonumber \\
%&\leq B_{SG}\sum_{j \in \Gamma_0} N_{j, i_{m+1}} \nu\abs{\nabla_jf}^q + B_{SG}\sum_{j \in (\Gamma_0)^c} N_{j, i_{m+1}} \nu\abs{\nabla_j f_m}^q 
%\label{weakSGIGamma1}
%\end{align}
\begin{equation}\label{weakSGIGamma1}
\begin{aligned}
\nu\bbE_{i_{m+1}}^\omega\abs{f_m - \bbE_{i_{m+1}}^\omega f_m}^q &\leq A\sum_{j \in \bbZ^D} M_{j, i_{m+1}} \nu\abs{\nabla_j f_m}^q \\
&\leq A\sum_{j \in \Gamma_0} M_{j, i_{m+1}} \nu\abs{\nabla_jf}^q + A\sum_{j \in (\Gamma_0)^c} M_{j, i_{m+1}} \nu\abs{\nabla_j f_m}^q,
\end{aligned}
\end{equation}
since there is no self-interaction in $\Gamma_0$ and $f_m$ depends only on coordinates in $(\Gamma_0)^c$ so that $\nabla_jf_m = (\nabla_jf)_m$ whenever $j \in \Gamma_0$. (Jensen's inequality then provides the final simplification.) For the summation over $(\Gamma_0)^c$, we apply the recursion \eqref{f2dInductionWYZ3} on each $\nu\abs{\nabla_jf_m}$ which gives 
\begin{equation}\label{weakSGIGamma2}
\nu\abs{\nabla_jf_m}^q \leq L_1\nu\abs{\nabla_jf}^q + c_1\sum_{k \in R_j}\nu\abs{\nabla_kf}^q + c_1L_2\sum_{k \notin R_j} M_{k, Q_j} \nu\abs{\nabla_kf}^q
\end{equation}
(Note $m \rightarrow \infty$ here but recall there is a cube $Q_j$ on which $\nu\abs{\nabla_j(\bbE_{\Gamma_0}f)}^q \leq \nu\abs{\nabla_j(\bbE_{Q_j}f)}^q$.) We can insert \eqref{weakSGIGamma2} into \eqref{weakSGIGamma1} and again into \eqref{weakSGIGamma0} to obtain
\begin{align*}
\frac{1}{2A}\nu\abs{f - \bbE_{\Gamma_0 f}}^q &\leq L_1\sum_{m=0}^\infty \sum_{j \in \bbZ^D} M_{j, i_{m+1}} \nu\abs{\nabla_jf}^q \\
&+ c_1L_2\sum_{m=0}^\infty \sum_{j \in (\Gamma_0)^c} M_{j, i_{m+1}} \sum_{k \in R_j} \nu\abs{\nabla_kf}^q \\ 
&+ c_1L_2\sum_{m=0}^\infty \sum_{j \in (\Gamma_0)^c} M_{j, i_{m+1}} \sum_{k \notin R_j} M_{k, Q_j} \nu\abs{\nabla_kf}^q 
\end{align*} 
since $L_1 > 1$. 

Let us denote by $I_1, I_2, I_3$ respectively the three terms. We note as in the previous proof that $\sum_{m=0}^\infty M_{j, i_{m+1}} \leq \sum_{j \in \bbZ^D} M_{i, j} \leq L_0$ uniformly in $j$ by \bf{(B2)} since $(M_{ij})$ is symmetric so column sums coincide with row sums, so $I_1$ is estimated by $I_1 \leq L_1L_0\nu\abs{\nabla f}^q$. $I_2$ and $I_3$ can be estimated with the arguments in the previous proof; the latter is entirely analogous to \eqref{sweepoutI2} and \eqref{sweepoutI3} but with the extra appearance of $\sum_{m=0}^\infty M_{j, i_{m+1}} \leq L_0$ so
\begin{align*}
I_2 &\leq c_1L_0L_2L_3\nu\abs{\nabla f}^q, \\
I_3 &\leq c_1L_0L_2L_4\nu\abs{\nabla f}^q.
\end{align*}
Collecting everything, we arrive at the weak $q$-SGI 
\[ \nu\bbE_{\Gamma_0}\abs{f - \bbE_{\Gamma_0}f}^q \leq C_0\abs{\nabla f}^q \]
with $C_0 = 2A(L_1 + c_1L_0L_2(L_3 + L_4)) = \cO(A) = \cO(\beta^{-1/2})$.
\end{proof}

Having obtained a weak $q$-SGI for the sub-lattices, we give a final argument to show this implies the (strong) $q$-SGI on the entire lattice. 

\begin{lemma}\label{lem:P}
Let $\cP \vcentcolon= \bbE_{\Gamma_{N-1}} \cdots \bbE_{\Gamma_0}\bbE_{\Gamma_0}$. Then $\cP$ satisfies the following four conditions: 
\begin{enumerate}[label={\normalfont\bfseries(\arabic*)}]
\item $\cP$ satisfies the DLR equation 
\[ \nu f = \nu(\cP f). \]
\item There exists a constant $C_1$ such that
\[ \nu\abs{f}^q - \nu\abs{\cP f}^q \leq C_1\nu\abs{\nabla f}^q, \]
\item There exists $\xi \in (0, 1)$ such that for any $m \in \bbZ_{\geq 0}$ one has 
\[ \nu\abs{\nabla(\cP^m f)}^q \leq \xi^m\nu\abs{\nabla f}^q. \]
\item The following limit
\[ \lim_{m \to \infty} \cP^m f = \nu f \]
holds $\nu$-almost surely.
\end{enumerate}
\end{lemma}

\begin{proof}
\bf{(1)} follows immediately from the DLR equation. To prove \bf{(2)}, denote by $f_{-1} = f$ and $f_n = \bbE_{\Gamma_n} \cdots \bbE_{\Gamma_0} \bbE_{\Gamma_0}f$, for $n = 0, 1, \cdots, N-1$. (This is a departure from the usual notation since the index of the sub-lattices start at zero to distinguish the index from their points, which start at one.) By applying the sub-quadratic inequality appearing in the previous proof and the weak sub-lattice $q$-SGI \eqref{weakSGIGamma} we obtain
\begin{align*}
\nu\abs{f}^q - \nu\abs{\cP f}^q &= \sum_{n=0}^{N-1} \nu\bbE_{\Gamma_n}\left(\abs{f_{n-1}}^q - \abs{\bbE_{\Gamma_n}f_{n-1}}^q\right) \\
&\leq 2\sum_{n=0}^{N-1} \nu\bbE_{\Gamma_n}\abs{f_{n-1} - \bbE_{\Gamma_n}f_{n-1}}^q \\
&\leq 2C_0 \sum_{n=0}^{N-1} \nu\abs{\nabla f_{n-1}}^q \\
&= 2C_0\nu\abs{\nabla f}^q + 2C_0\sum_{n=0}^{N-2} \nu\abs{\nabla_{(\Gamma_n)^c}(\bbE_{\Gamma_n} f_{n-1})}^q
\end{align*}
%so that by the weak sub-lattice $q$-SGI \eqref{weakSGIGamma} we obtain 
%\[ \nu\abs{f}^q - \nu\abs{\cP f}^q \leq 2C_0 \sum_{n=0}^{N-1} \nu\abs{\nabla f_{n-1}}^q = 2C_0\nu\abs{\nabla f}^q + 2C_0\sum_{n=0}^{N-2} \nu\abs{\nabla_{(\Gamma_n)^c}(\bbE_{\Gamma_n} f_{n-1})}^q \]
since $f_n = \bbE_{\Gamma_n} f_{n-1}$ depends only on coordinates in $(\Gamma_n)^c$. The first term requires no further estimation, whilst we can apply Lemma \ref{lem:sweepoutGamma} inductively on the second term to obtain 
\[ \nu\abs{\nabla_{(\Gamma_n)^c}(\bbE_{\Gamma_n} f_{n-1})}^q \leq K_2^{n+1}\nu\abs{\nabla f}^q \]
which gives 
\[ \nu\abs{f}^q - \nu\abs{\cP f}^q \leq C_1\nu\abs{\nabla f}^q \]
which is exactly \bf{(2)} with $C_1 = 2C_0\sum_{n=0}^{N-1} K_2^n \leq 2K_2^NC_0$. 

To prove \bf{(3)}, we define $\Gamma^{(m)} = \cup_{n=0}^m \Gamma_n$ and apply Lemma \ref{lem:sweepoutGamma} again to obtain 
\[ \nu\abs{\nabla_{\Gamma^{(m)}} f_m}^q = \nu\abs{\nabla_{\Gamma^{(m-1)}} f_m}^q \leq K_2\nu\abs{\nabla_{\Gamma^{(m-1)} } f_{m-1}}^q + c_2\nu\abs{\nabla f_{m-1}}^q. \]
Applying Lemma \ref{lem:sweepoutGamma} as in the proof for \bf{(2)}, we obtain 
\begin{equation}\label{ind1}
\nu\abs{\nabla_{\Gamma^{(m)}} f_m}^q \leq K_2\nu\abs{\nabla_{\Gamma^{(m-1)} } f_{m-1}}^q + c_2K_2^m\nu\abs{\nabla f}^q.
\end{equation}
This equation can be inductively applied on the first term, that is applying \eqref{ind1} on $\nu\abs{\nabla_{\Gamma^{(m-1)} } f_{m-1}}^q$ and starting from $m = N-1$ we find 
\[ \nu\abs{\nabla(\cP f)}^q \leq K_2^N \nu\abs{\nabla_{\Gamma_0} f_0}^q + c_2K_2^N \nu\abs{\nabla f}^q = c_2K_2^N\nu\abs{\nabla f}^q \]
since $\nabla_{\Gamma_0}f_0 = \nabla_{\Gamma_0}\bbE_{\Gamma_0}f = 0$. Since $K_2 = \cO(1)$ and $c_2 = o(1)$, we obtain \bf{(3)} at $m = 1$ with $\xi = c_2K_2^N$ with $\beta$ small enough; the general case for $m \in \bbZ_{\geq 1}$ follows by replacing $f$ by $\cP f$ and proceeding by straightforward induction. 

To prove the final statement \bf{(4)}, following the arguments of \bf{(2)} we first bound the $q$-variance with respect to $\cP$ by
% https://math.stackexchange.com/questions/1070961/why-this-power-inequality-for-sums-of-real-numbers-holds
\[ \nu\abs{f - \cP f}^q \leq N^N\sum_{n=0}^{N - 1} \nu\bbE_{\Gamma_n} \abs{f_{n-1} - \bbE_{\Gamma_n} f_{n-1}}^q \leq N^NC_1\nu\abs{\nabla f}^q \]
using the H\"older inequality $(\sum_{i=1}^N \abs{x_i})^q \leq N^{q-1}\sum_{i=1}^N \abs{x_i}^q$ for nonnegative $(x_i)$ on the telescoping sum $f - \cP f = \sum_{n=0}^{N-1} f_{n-1} - f_n$. For $m \in \bbZ_{\geq 1}$, replacing $f$ by $\cP^mf$ in the above estimate we obtain 
\[ \nu\abs{\cP^mf - \cP^{m+1}f}^q \leq N^NC_1\nu\abs{\nabla(\cP^mf)}^q \leq N^NC_1\xi^m\nu\abs{\nabla f}^q. \]
Since $\xi < 1$, the Borel-Cantelli lemma implies $\cP^m$ converges $\nu$-almost surely to some function $g$, and similarly $\abs{\nabla(\cP^mf)}$ converges $\nu$-almost surely to zero by \bf{(3)}, so that $g$ is constant $\nu$-almost everywhere. Note the decay of $C_1$ with respect to $\beta$ does not matter since $\beta$ is sufficiently small, but fixed. Assuming that $f \in W^{1,p}(\nu)$ is bounded, by dominated convergence we have
\[ \lim_{m \to \infty} \cP^m f = g = \nu(g) = \lim_{m \to \infty} \nu(\cP^m f) = \nu(f) \]
with the last equality obtained by \bf{(1)}. The general case of possibly unbounded $f \in W^{1,p}(\nu)$ is obtained by density. 
% might need more here?
\end{proof}

We can now prove the main theorem of this section. 

\begin{proof}[Proof of Theorem \ref{thm:singletoglobal}]
Denote by $f^m = \cP^mf$ for $m \in \bbZ_{\geq 0}$. By applying \bf{(1)}, \bf{(2)}, and \bf{(3)} of Lemma \ref{lem:P} in succession, 
\begin{align*}
\nu\abs{f - \nu f}^q &= \sum_{n=0}^{N-1}\left(\nu\abs{f_n - \nu f}^q - \nu\abs{\cP(f_k - \nu f)}^q\right) + \nu\abs{f_N - \nu f}^q \vphantom{C_1\left(\sum_{n=0}^{N-1} \xi^n\right)\nu\abs{\nabla f}^q + \nu\abs{f_n - \nu f}^q} \\
&\leq C_1\sum_{n=0}^{N-1} \nu\abs{\nabla f_n}^q + \nu\abs{f_N - \nu f}^q \vphantom{C_1\left(\sum_{n=0}^{N-1} \xi^n\right)\nu\abs{\nabla f}^q + \nu\abs{f_n - \nu f}^q} \\
&\leq C_1\left(\sum_{n=0}^{N-1} \xi^n\right)\nu\abs{\nabla f}^q + \nu\abs{f_n - \nu f}^q \\
&\leq \frac{JC_1}{1 - \xi}\nu\abs{\nabla f}^q + \nu\abs{f_N - \nu f}^q. \vphantom{C_1\left(\sum_{n=0}^{N-1} \xi^n\right)\nu\abs{\nabla f}^q + \nu\abs{f_n - \nu f}^q}
\end{align*}
Taking $N \rightarrow \infty$, the remainder term satisfies $\nu\abs{f_N - \nu f}^q \rightarrow 0$ by \bf{(4)} which gives the global $q$-SGI 
\[ \nu\abs{f - \nu f}^q \leq C_{SG}\nu\abs{\nabla f}^q \]
with $C_{SG} = C_1/(1 - \xi)$. 
\end{proof}

\subsection{Generalisations to interactions of finite range $R$}
%Section 3.1%
The analysis can be generalised to the case where the interaction has finite range $R > 1$, meaning $\beta_{ij} \equiv 0$ for $\norm{i - j} > R$. No modification to the main theorem in this case is necessary: one modifies the definition $j \sim i$ to mean $1 \leq \norm{j - i} \leq R$, i.e. $j$ is a neighbour when $j$ is allowed to interact with $i$, and much of the analysis is left unchanged modulo changes to $\cO(1)$ estimates. In Lemma \ref{lem:sweepoutineq}, the only dependence on the range of the interaction (beyond the fact the definition $j \sim i$ depends on the range) appears in the definition of $\epsilon$. When $R = 1$, the number of neighbours of $i \in \bbZ^D$ is exactly $2D$; when $R > 1$ the number of neighbours grows like $R^D$ but for fixed $R$ is finite, so that we can replace $\epsilon = 2da_2$ with $\epsilon = N_Ra_2$ where $N_R = \abs{\{\sim i\}}$ is the number of neighbours of $i$. 

A generalisation of the sub-lattices $(\Gamma_n)_{n=0}^{N-1}$ is also straightforward: it is easy to see the base cube $\Gamma_0 = (R + 1)\bbZ^D$ has no self-interaction and that there are $N_1 = (R + 1)^D$ translations of $\Gamma_0$ shifting the basepoint $0$ to another point in $[0, R]^D \sbs \bbZ^D$ partition $\bbZ^D$. Note one needs to make a distinction between the number $N_1$ of sub-latties in our partition and the number $N_2 = 2^D$ of vertices in a $D$-dimensional cube, which coincide for $R = 1$. In any case, only $N_2$ has a role to play in the proof of Lemmas \ref{lem:sweepoutGamma} and \ref{lem:weakSGIGamma}, where the sub-lattice is fixed (and so the number of sub-lattices is not relevant) so we only need to change the estimate on $\abs{R_j} = (2R + 3)^D = \cO(1)$, which makes no difference to the validity of the proof because its influence is only felt in the $\cO(1)$ constants $L_2, L_3, L_4$. Similarly the $N$ appearing in the proof of parts \bf{(3)} and \bf{(4)} in Lemma \ref{lem:P} is actually $N_1$ and not $N_2$, but again its influence is $\cO(1)$ and can be dealt with by taking $c_2 = o(1)$ sufficiently small.
%Before we conclude this section we remark the requirements of Theorem \ref{thm:singletoglobal} may be relaxed, but we have not done so for the sake of clearer exposition. The assumption of a weak $U$-bound and a weak $q$-SGI is necessary, but we can generalise the assumptions on $A$ by giving \bf{(A2)} a different constant $B$ say and requiring only that $AB = \cO(\beta^{-1})$, which is required in the proof of Lemma \ref{lem:sweepoutineq} to ensure $a_2 = o(1)$. Similarly the symmetry of $(M_{ij})$ is used only to give uniform boundedness of the $L^1$-sum in both coordinates 
%
%This assumption in turn can be relaxed since the proof only requires that $\frac{a_2 + a_1a_2}{1 - \epsilon}$ is finite, so if one has stronger bounds on $M_{ki}$ then $\epsilon = 2da_2
%
%On the other hand, if an estimate on exactly how small $\beta$ must be in order for the conditions of Theorem \ref{thm:singletoglobal} to be satisfied is desirable, we may impose additional restrictions; in the next section, we see that the generic case 

\section{Sufficient conditions on the phase and interaction}
%Section 4%
Having now given general conditions - the existence of a weak single-site $U$-bound and a weak single-site $q$-SGI - which imply the global $q$-SGI, we now consider conditions on the phase $\phi \in C(\bbG)$ and interaction $V \in C^1(\bbG \times \bbG)$ appearing in \eqref{def:potential} such that the conditions of Theorem \ref{thm:singletoglobal} hold. The first is a (strong) single-site $U$-bound without interaction with respect to a function $\eta$, and the second is a condition on the interaction given in terms of $\eta$. 

\begin{theorem}\label{thm:SGIglobalC1C2}
%%Section 4 Theorem 4.1
Let $q \in (1, 2]$, and let $\nu$ be a Gibbs measure for the local specification $(\bbE_\Lambda^\omega)_{\Lambda \Subset \mathbb{Z}^D}^{\omega \in \mathcal{S}}$ given in \eqref{def:potential}. Let $f \in W^{1,q}(\nu)$ and suppose the following two conditions hold on $\phi$ and $V$ for sufficiently small $\beta > 0$: 
\begin{enumerate}[label=\normalfont{\bf{(C\arabic*)}}]
\item There exists a constant $B > 0$ and a function $\eta \in C(\bbG; \bbR_{\geq 0})$ diverging to infinity all directions such that 
\[ \int \abs{g}^q \eta(h) e^{-\phi(h)} dh \leq B\int \left(\abs{\nabla g}^q + \abs{g}^q\right) e^{-\phi(h)}dh \]
for all $g \in C(\bbG)$ for which the LHS is finite, and 
% and moreover
%\[ \sup_{i \in \bbZ^D} \nu(\eta(x_i)\abs{f}^q) = \sup_{i \in \bbZ^D} \nu\bbE_i^\omega(\eta(x_i)\abs{f}^q) < \infty. \]
\item the interaction $V \in C(\bbG \times \bbG)$ satisfies the gradient bound
\[ \abs{\nabla_{h_1} V(h_1, h_2)}^q + \abs{\nabla_{h_2}V(h_1, h_2)}^q \leq B(1 + \eta(h_1) + \eta(h_2)). \]
%\item The function $\eta_i\abs{f}^q$ is uniformly $\nu$-integrable in $i$, that is
%\[ \sup_{i \in \bbZ^D} \nu(\eta(x_i)\abs{f}^q) = \sup_{i \in \bbZ^D} \nu\bbE_i^\omega(\eta(x_i)\abs{f}^q) < \infty. \]
\end{enumerate}
Then for all sufficiently small $\beta > 0$ the global $q$-SGI 
\[ \nu\abs{f - \nu f}^q \leq C_{SG}\nu\abs{\nabla f}^q \]
holds for all $f \in W^{1,q}(\nu)$ satisfying $\sup_{i \in \bbZ^D} \nu(\eta(x_i)\abs{f}^q) < \infty$ with some constant $C_{SG} > 0$. 
\end{theorem}

\begin{remark}
The reader will note that what follows below at the expense of essentially notational complications could be generalised to more general multispin superstable interaction in the sense of \cite{Ru70, LePr76}.
\end{remark}

Note that although this setup allows for more general $U$-bounds, the generic case occurs when the phase, the interaction, and the $U$-bound are all powers of a homogeneous norm. For instance, in the case of the Carnot-Carath\'eodory distance $d$, \cite{InPa09} proves the $q$-logarithmic Sobolev inequality with phase $\phi = d^p$, interaction $V(x, y) = (d(x) \pm d(y))^2$, and $U$-bound $\eta = d^p$, in the case $p \geq 2$. This setup allows for more general $U$-bound functions $\eta$ which may exist outside the generic case and asserts the only requirement of $\eta$ is that it controls the $q$-th exponent of the sub-gradient of the interaction in both coordinates.

The main problem in the proof is to show that \bf{(C1)} and \bf{(C2)}, which we assume hold in the sequel, imply the existence of a weak $q$-SGI. Since $\abs{\nabla_j V}^q$ is controlled by $\eta$, the weak $U$-bound of \bf{(A1)} follows immediately, and the difficulty lies with \bf{(A2)}. In \cite[Theorem~3.1]{HeZe09} it was shown in the finite dimensional case that one can pass from a $U$-bound to a $q$-SGI, and in \cite[Lemma~4.3]{InPa09} it was shown in some class of models that one can pass from a strong single-site $U$-bound to the strong $q$-SGI. Here we show the statement remains true with ``weak'' replacing ``strong''. 

The first step in establishing the weak $U$-bound for $\eta$ is to prove a weak $U$-bound for $\eta_i$ that is defective in the sense the RHS contains copies of the weak $U$-bound for neighbouring $\eta_j$. We will then approach in the spirit of the proofs in \S3 whereby we prove that inductively inserting the $U$-bound into these defective copies ``converges'' appropriately. It will also be noted as a corollary that the proof of this weak $U$-bound will give the existence of the Gibbs measure by a criterion of Dobrushin. In the sequel, denote for convenience $\eta_i \vcentcolon= \eta(x_i)$ the $U$-bound function $\eta$ in the $x_i$-coordinate, and similarly denote $\phi_i \vcentcolon= \phi(x_i)$ and $V_{ij} \vcentcolon= V(x_i, x_j)$, with the understanding $x_j = \omega_j$ whenever $x_j$ is left unspecified.

\begin{lemma}\label{lem:dUboundnonuniform}
%Section 4 Lemma 4.1%
Assume for $q\in[1,\infty)$ we have
\[\beta^q\vcentcolon=\sup_{i}\sum_{j\neq i} \abs{\beta_{ij}}^q< \infty 
\]
and
\[
2^{q-1}B^2\beta^q <1.
\]
Then for all $i \in \bbZ^D$,
\begin{equation}\label{dUboundnonuniformInduction} \bbE_i^\omega(\eta_i\abs{f}^q) \leq K_1\,\left( \bbE_i^\omega\abs{\nabla_if}^q + \bbE_i^\omega\abs{f}^q\right) +c_1\, \sum_{j \neq i} \abs{\beta_{ij}}^q\; \bbE_i^\omega(\eta_j\abs{f}^q ) .
\end{equation}
with
\[\begin{split}
    K_1 &\vcentcolon= \frac{\max\left\{2^{q-1}B,\, B(1+2^{q-1}B\beta^q) \right\}}{(1 - 2^{q-1}B^2\beta^q)}
    \\
    c_1 &\vcentcolon= \frac{2^{q-1}B^2}{(1 - 2^{q-1}B^2\beta^q)}
\end{split}\]
Hence for $f \equiv 1$
\begin{equation}\label{6.2} 
    \bbE_i^\omega(\eta_i ) \leq K_1\,  +c_1\, \sum_{j \neq i} \abs{\beta_{ij}}^q\;  \eta_j(\omega).
\end{equation}
\end{lemma}

\begin{proof}
%By assumption that $\sup_{i \in \bbZ^D} \nu(\eta_i\abs{f}^q) < \infty$,
For a compactly supported differentiable function $f$ and under the assumption that the $V_{ij}$ are a.e. bounded on compact sets and weakly differentiable with derivatives a.e., we can replace $g$ with $\smash{fe^{-\frac{1}{q}\sum_{j \neq i} \beta_{ij}V_{ij}}}$ in \bf{(C1)}, to obtain 
\begin{align*}
\bbE_i^\omega&(\eta_i\abs{f}^q)\\
&= \int \eta_i\abs{f}^q   e^{-U_i^\omega} dx_i  \leq B\int \abs{\nabla_if - \frac{1}{q}f\sum_{j \neq i} \beta_{ij}\nabla_iV_{ij}}^q e^{-U_i^\omega} dx_i + B\int\abs{f}^q e^{-U_i^\omega} dx_i \\
&\leq 2^{q-1}B\int\abs{\nabla_if}^q e^{-U_i^\omega} dx_i + 2^{q-1}B\sum_{j \neq i} \abs{\beta_{ij}}^q\int \abs{f}^q\abs{\nabla_iV_{ij}}^q e^{-U_i^\omega} dx_i
\vphantom{\abs{\nabla_if - \frac{1}{q}f\sum_{j \sim i} \beta_{ij}\nabla_iV_{ij}}^q} 
+ B\int\abs{f}^q e^{-U_i^\omega} dx_i ,
\end{align*}
that is 
\[ \bbE_i^\omega(\eta_i\abs{f}^q) \leq 2^{q-1}B\, \bbE_i^\omega\abs{\nabla_if}^q + 2^{q-1}B\,   \sum_{j \neq i} \abs{\beta_{ij}}^q\bbE_i^\omega(\abs{f}^q\abs{\nabla_iV_{ij}}^q) + B\,\bbE_i^\omega\abs{f}^q. \]
Applying \bf{(C2)} to the second term, we have
\[ \bbE_i^\omega(\eta_i\abs{f}^q) \leq 2^{q-1}B\, \bbE_i^\omega\abs{\nabla_if}^q + 2^{q-1}B^2\, \sum_{j \neq i} \abs{\beta_{ij}}^q\; \bbE_i^\omega(\abs{f}^q\left(1+\eta_i+\eta_j\right) )+ B\,\bbE_i^\omega\abs{f}^q. \]
and after rearranging, with 
$\beta^q\equiv\sup_{i}\sum_{j\neq i} \abs{\beta_{ij}}^q$, we get
\begin{align*}
    (1 - 2^{q-1}B^2\beta^q)\bbE_i^\omega(\eta_i\abs{f}^q) &\leq 2^{q-1}B\, \bbE_i^\omega\abs{\nabla_if}^q + 2^{q-1}B^2\, \sum_{j \neq i} \abs{\beta_{ij}}^q\; \bbE_i^\omega(\eta_j\abs{f}^q ) \\
    &+ B(1+2^{q-1}B\beta^q)\,\bbE_i^\omega\abs{f}^q.
\end{align*}
Now this relation can be generalised to functions which are not compactly supported. In this way the relation \eqref{dUboundnonuniformInduction} follows.
The relation \eqref{6.2} is the special case for $f$ being the unit function.

\begin{comment}
 Hence we find through the DLR equation
\[ (1 - 4dB^2\beta^q)\nu(\eta_i\abs{f}^q) \leq 2B\nu\abs{\nabla_if}^q + B(1 + 4dB\beta^q)\nu\abs{f}^q + 2B^2\beta^q\sum_{j \sim i} \nu(\eta_j\abs{f}^q). \]
Since $B = \cO(1)$, for $\beta$ small enough we obtain \eqref{dUboundnonuniformInduction} with $K_1 = 4B = \cO(1)$ and $c_1 = \beta = o(1)$. 
%\[ (1 - 4dA_1A_2\beta^q)\nu(\eta_i\abs{f}^q) \leq 2A_1\nu\abs{\nabla_if}^q + A_1(1 + 4dA_2\beta^q)\nu\abs{f}^q + 2A_1A_2\beta^q\sum_{j \sim i} \nu(\eta_j\abs{f}^q). \]
%(We implicitly divide by the normalisation constant $Z_i^\omega = \int e^{U_i^\omega} dx_i$ so that the integrals represent the measure $\bbE_i^\omega$, then apply $\nu$ together with the DLR equation.) If $\beta$ is sufficiently small then $4dA_1A_2\beta^q < 1$ and
%\[ \nu(\eta_i\abs{f}^q) \leq J\left(\nu\abs{\nabla_if}^q + \nu\abs{f}^q\right) + c_\beta \sum_{j \sim i} \nu(\eta_j\abs{f}^q) \]
%where $J = \frac{2A_1 + A_1(1 + 4dA_2\beta^q)}{1 - 4dA_1A_2\beta^q}$ and $c_\beta = \frac{2A_1A_2\beta^q}{1 - 4dA_1A_2\beta^q}$, which is exactly \eqref{dUboundnonuniformInduction}, and obviously $J\beta, c_\beta \rightarrow 0$ as $\beta \rightarrow 0$.
\bigskip
Hence the following defective weak single-site $U$-bound 
\begin{equation}\label{dUboundnonuniformInduction}
\nu(\eta_i \abs{f}^q) \leq K_1\left(\nu\abs{\nabla_if}^q + \nu\abs{f}^q\right) + c_1 \sum_{j \neq i} \nu(\eta_j \abs{f}^q)
\end{equation}
holds for some constants $K_1 = \cO(1)$ and $c_1 = o(1)$.
\end{comment}
\end{proof}
%\bigskip 

Although the $U$-bounds were introduced in \cite{HeZe09} and used to obtain coercive inequalities on, in particular, the Heisenberg group, it turns out the $U$-bounds can be used to play the role of the bicompact function of Dobrushin's existence theorem to show that a Gibbs measure for \eqref{def:potential} does in fact exist. We first restate Dobrushin's existence theorem (with notation adapted to the setting presented here) and then we show the conditions are satisfied. 

\begin{theorem}
For the existence of a Gibbs measure $\nu$ associated to the local specification $(\bbE_\Lambda^\omega)_{\Lambda \sbs \bbZ^D}^{\omega\in\mathcal{S}}$, the fulfilment of the following two conditions is sufficient:
\begin{enumerate}
    \item There exists a bicompact function $\eta: \bbG \rightarrow \bbR$ and constants $C, c_{i, j}, c \geq 0$ such that 
    \[ \int_\bbG \eta(x)\bbE_i^\omega(dx) \leq C + \sum_{j \neq i} c_{i, j}\eta(\omega_j) \]
    for all $i \in \bbZ^D$, $j \neq i$, and all $\omega \in \cS$, and 
    \[ \sum_{j \neq i} c_{i, j} \leq c < 1. \]
    \item For all $i \in \bbZ^D$ there exists a sequence of finite sets $\Lambda_i^1 \sbs \Lambda_i^2 \sbs \cdots$ whose union is $\bbZ^D \setminus \{i\}$ and constants $D_n, d_{i, j}^n \geq 0$ such that
    \[ \sum_{j \neq i} d_{i, j}^n \leq D_n \]
    for all $j \neq i$ and $n = 1, 2, \cdots$ with the property that for each bounded continuous function $f: \bbG \rightarrow \bbR$ there exists a sequence of continuous functions $(f_n: \Omega \rightarrow \bbR)_{n \in \bbN}$ depending only on sites $x_k$ for which $k \in \Lambda_i^n$ such that 
    \[ \abs{\int_\bbG f(x)\bbE_i^\omega(dx) - f_n(\omega)} \leq D_n + \sum_{j \neq i} d_{i, j}^n\eta(\omega_j) \]
    for all $\omega \in \cS$. 
\end{enumerate}
\end{theorem}

\begin{corollary}
%Cor4.2
There exists a Gibbs measure $\nu$ associated to the local specification $(\bbE_\Lambda^\omega)_{\Lambda \sbs \bbZ^D}^{\omega\in\mathcal{S}}$.
\end{corollary}

\begin{proof}
The first condition is precisely \eqref{6.2},
\[ \bbE_i^\omega(\eta_i) \leq K_1 + c_1\sum_{j\neq i}\abs{\beta_{ij}}^q\eta_j(\omega), \]
with $C = K_1$ and $c_{i, j} = c_1\abs{\beta_{ij}}^q$ which, for $\beta = \sup_{i,j} \abs{\beta_{ij}}$ sufficiently small, can be made uniformly small. The second condition on the other hand asks whether the conditional expectation $\bbE_i^\omega(\eta_i)$ may be approximated by a sequence $(f_n)$ of functions depending on finitely many spins $\omega_i$, but this is vacuous in the setting of finite range interactions. 
\end{proof}

\begin{remark}
One can expect that by using more space and time it should be possible to extend the compactness considerations of \cite{BeHK82} to the present context and avoid our restriction on $0<\beta$ being small.
\end{remark}

We now continue with the proof of Theorem \ref{thm:SGIglobalC1C2}. As mentioned earlier, the next step uses the inductive arguments of \S3 to show the defective term $c_1 \sum_{j \sim i} \nu(\eta_j\abs{f}^q)$ can be controlled. 

\begin{lemma}\label{lem:dUbound}
%%Lem 4.3
For all $i \in \bbZ^D$, the following weak single-site $U$-bound 
\begin{equation}\label{dUbound}
\nu(\eta_i\abs{f}^q) \leq K_2\left(\nu\abs{f}^q + \sum_{k \in \bbZ^D} c_2^{\norm{k - i}} \nu\abs{\nabla_kf}^q\right)
\end{equation}
holds for some constants $K_2 = \cO(1)$ and $c_2 = o(1)$. 
\end{lemma}

\begin{proof}
For notational convenience, denote by 
\[ W_i \vcentcolon= \nu(\eta_i\abs{f}^q), \quad Y_i \vcentcolon= \nu\abs{\nabla_if}^q, \quad Z = \nu\abs{f}^q \]
so that \eqref{dUboundnonuniformInduction} reads 
\begin{equation}\label{dUboundInductionRelation}
W_i \leq K_1Y_i + K_1Z + c_1 \sum_{k \sim i} W_k
\end{equation}
whilst \eqref{dUbound} reads $W_i \leq K_2Z + K_2\sum_{k \in \bbZ^D} c^{\norm{k - i}}Y_k$. The difference between the induction argument here and that which was given in \S3 is that in the previous case case the LHS depended on $j$ while the defective on the RHS was a summation over a set which depended on $i$. This is in contrast to the current case where the LHS \emph{also} depends on $i$. This means that in the first case only terms depending on the neighbours of $i$ appear whereas when we apply \eqref{dUboundInductionRelation} to each $W_k$ in the sum we expect terms depending on neighbours of neighbours of $i$ to appear, and iterating this procedure eventually each term $\nu\abs{\nabla_kf}^q$ should appear. 

Indeed, let us look at 
\begin{equation}\label{iteration1}
\sum_{k \sim i} W_k \leq 2D K_1 Z + K_1 \sum_{k \sim i} Y_k + c \sum_{k \sim i} \sum_{\ell \sim k} W_\ell,
\end{equation}
and in general if the nested summation is denoted by 
\[ \cS_M(W_i) \vcentcolon= \sum_{k_1 \sim i} \sum_{k_2 \sim k_1} \cdots \sum_{k_M \sim k_{M-1}} W_{k_M}, \]
then
\begin{equation}\label{iterationgeneral}
\cS_M(W_i) \leq K_1\cS_M(Z) + K_1\cS_M(Y_i) + c \cS_{M + 1}(W_i).
\end{equation}
where $\cS_M(Z)$ and $\cS_M(Y_i)$ are defined similarly. The first term $J\cS_M(Z)$ is easily handled, since $Z$ has no dependence on any indices so 
\[ \cS_M(Z) = \sum_{k_1 \sim i} \sum_{k_2 \sim k_1} \cdots \sum_{k_M \sim k_{M-1}} Z = (2D)^MZ. \]
The second term $\cS_M(Y_i)$ requires some counting: the summation contains, for instance, terms $Y_k$ for which $\norm{k - i} = M$, but these terms can only occur in the innermost summation, of which there are $2D$ terms, so a crude estimate is that the contribution of those $Y_k$ for which $\norm{k - i} = M$ is controlled by $2D\sum_{\norm{k - i} = M} Y_k$. Similarly, terms $Y_k$ for which $\norm{k - i} = M-1$ can only occur in the two innermost summations, of which there are $(2D)^2$ terms, and so on so forth; continuing in this way we find 
\begin{equation}\label{estimateonY}
\cS_M(Y_i) \leq \sum_{n=0}^{M} \sum_{\norm{k - i} = n} (2D)^{M-n}Y_k \leq (2D)^M \sum_{n=0}^M \sum_{\norm{k - i} = n} Y_k.
\end{equation}
The final term $\cS_{M + 1}(W_i)$ is handled by our initial assumption $\sup_{i \in \bbZ^D} \nu(\eta_i\abs{f}^q) = \sup_{i \in \bbZ^D} W_i < \infty$, implying each $W_k \leq L$ for some constant $L > 0$ uniformly in $k$. We can then crudely estimate the number of terms by $(2D)^{M+1}$ so that $\cS_{M+1}(W_i) \leq (2D)^{M+1}L$. 

With $\cS_0(X) = X$ where $X$ is one of $W_i, Y_i, Z$, we find 
\begin{align*}
W_i &\leq K_1\cS_0(Y_i) + K_1\cS_0(Z) + c_1 \cS_1(W_i) \vphantom{K_1\sum_{r=0}^M c_1^r \cS_r(Y_i) + K_1\sum_{r=0}^M c_1^r \cS_r(Z) + c_1^{M+1} \cS_{M+1}(W_i)} \\
&\leq K_1\cS_0(Y_i) + K_1\cS_0(Z) + c_1 K_1\cS_1(Z) + c_1 K_1\cS_1(Y_i) + c_1^2 \cS_2(W_i) \vphantom{K_1\sum_{r=0}^M c_1^r \cS_r(Y_i) + K_1\sum_{r=0}^M c_1^r \cS_r(Z) + c_1^{M+1} \cS_{M+1}(W_i)} \\
&\mathrel{\makebox[\widthof{=}]{\vdots}} \vphantom{K_1\sum_{r=0}^M c_1^r \cS_r(Y_i) + K_1\sum_{r=0}^M c_1^r \cS_r(Z) + c_1^{M+1} \cS_{M+1}(W_i)} \\
&\leq K_1\sum_{r=0}^M c_1^r \cS_r(Y_i) + K_1\sum_{r=0}^M c_1^r \cS_r(Z) + c_1^{M+1} \cS_{M+1}(W_i) \\
&\leq K_1\sum_{r=0}^M c_1^r \cS_r(Y_i) + K_1\sum_{r=0}^M (2Dc_1)^rZ + (2Dc_1)^{M+1}L
%&\leq K_1\sum_{r=0}^M c_1^r \cS_r(Y_i) + \frac{K_1}{1 - 2Dc_1}Z + (2Dc_1)^{M+1}L 
\end{align*}
so that if $\beta$ is small enough then $\epsilon \vcentcolon= 2Dc_1 < 1/2$, giving $K_1\sum_{r=0}^M (2Dc_1)^r Z \leq \frac{K_1}{1 - \epsilon} Z \leq 2K_1Z$ and $(2Dc_1)^{M+1}L \leq \epsilon^{M+1}L$. For the first term, our previous estimate \eqref{estimateonY} gives 
\begin{align*}
\sum_{r=0}^M c_1^r \cS_r(Y_i) &\leq \sum_{r=0}^M c_1^r (2D)^r \sum_{n=0}^{r} \sum_{\norm{k - i} = n} Y_k = \sum_{n=0}^M \sum_{r=n}^M (2Dc_1)^r \sum_{\norm{k - i} = n} Y_k \\
&\leq 2\sum_{n=0}^M (2Dc_1)^n \sum_{\norm{k - i} = n} Y_k  
\end{align*}
Thus we arrive at 
\[ W_i \leq 2\sum_{n=0}^M (2Dc_1)^n \sum_{\norm{k - i} = n} Y_k + 2K_1 Z + (2Dc_1)^{M+1}L \]
for all $M \in \bbZ_{\geq 0}$ and $\beta$ small enough. Taking $M \rightarrow \infty$ we conclude \eqref{dUbound} holds for $K_2 = \max(2, 2K_1) = \cO(1)$ and $c_2 = 2Dc_1 = \cO(c_1) = o(1)$. 
\end{proof}

We can now verify that the conditions of Theorem \ref{thm:singletoglobal} hold and prove the $q$-SGI. The key ingredient is the previous lemma - it is straightforward to see that the weak $U$-bound is immediate, and more difficult to see that the weak $U$-bound implies the weak $q$-SGI of \bf{(A2)}, but we show the argument to show a strong $U$-bound implies a strong $q$-SGI as in \cite[Lemma~4.3]{InPa09} can be adapted to the weak case. The basic idea is that one can truncate the LHS of \bf{(A2)} into two parts, one where the boundary conditions are bounded and the other where they are not. It turns out that the latter is simpler to handle as the $U$-bound can be inserted. 

\begin{proof}[Proof of Theorem \ref{thm:SGIglobalC1C2}]
%%%Thm 4.1
By applying Lemma \ref{lem:dUbound} and \bf{(C2)}, we obtain
\begin{align*}
\nu(\abs{f}^q \abs{\nabla_j V_{ij}}^q) &\leq B\nu\abs{f}^q + B\nu(\eta_i\abs{f}^q) + B\nu(\eta_j\abs{f}^q) \vphantom{A_2\sum_{k \in \bbZ^D} \left(C_\beta^{\norm{k - i}} + C_\beta^{\norm{k - j}}\right)\nu\abs{\nabla_kf}^q} \vphantom{\left(\nu\abs{f}^q + \sum_{k \in \bbZ^D} \left(c_2^{\norm{k - i}} + c_2^{\norm{k - j}}\right)\nu\abs{\nabla_kf}^q\right)} \\
&\leq B(1 + 2K_2)\nu\abs{f}^q + BK_2\sum_{k \in \bbZ^D} \left(c_2^{\norm{k - i}} + c_2^{\norm{k - j}}\right)\nu\abs{\nabla_kf}^q \vphantom{\left(\nu\abs{f}^q + \sum_{k \in \bbZ^D} \left(c_2^{\norm{k - i}} + c_2^{\norm{k - j}}\right)\nu\abs{\nabla_kf}^q\right)} \\
&\leq B(1 + 2K_2)\left(\nu\abs{f}^q + \sum_{k \in \bbZ^D} \left(c_2^{\norm{k - i}} + c_2^{\norm{k - j}}\right)\nu\abs{\nabla_kf}^q\right)
\end{align*}
This is \bf{(A1)} with $A = B(1 + 2K_2) = \cO(1)$ but with an incorrect interaction matrix $M_{ki}$ since $c_2^{\norm{k - i}} + c_2^{\norm{k - j}}$ has dependence on $j$. Defining $\dist(k, S) = \inf_{j \in S} \norm{k - j}$ for $S \sbs \bbZ^D$, we see that for $i_1 = \{j \in \bbZ^D \mid \dist(j, i) \leq 1\} = \{i\} \cup \{\sim i\}$ that 
\[ c_2^{\norm{k - i}} + c_2^{\norm{k - j}} \leq 2c_2^{\dist(k, i_1)} \]
since $j \sim i$ by assumption. Thus the correct interaction matrix is $M_{ki} = c_2^{\dist(k, i_1)}$ modulo a constant. Note $M_{ij} \geq 0$ for all $i, j \in \bbZ^D$, and $\sum_{j \in \bbZ^D} M_{ij}$ is clearly independent of $i$ and finite for $c_2 < 1$ since the number of sites with $\norm{i - j} = n$ is grows polynomially like $n^D$ while $M_{ij}$ for such sites decays as $c_2^n$, meaning $\sum_{j \in \bbZ^D} M_{ij}$ is controlled by $\sum_{n \geq 0} c_2^nn^D < \infty$. It follows \bf{(B2)} holds, but we hold off judgment on \bf{(B1)} since we may need to enlarge the constant $B(1 + 2K_2)$ depending on the constant we find for \bf{(A2)}. 

To obtain \bf{(A2)}, fix a  parameter $L_1 > 1$ and denote by $B(R) = \{g \in \bbG \mid d(g) \leq R\}$ the metric ball of radius $R$ in the Heisenberg group associated to the Carnot-Carath\'eodory distance $d$ on $\bbG$. Since $\eta$ is continuous and diverges to infinity in all directions, we can show there exists $R, L_2 > 1$ such that 
\begin{equation}\label{etainballs}
\{\eta_i \leq L_1\} \sbs B(R) \sbs B(3R) \sbs \{\eta_i \leq L_2\}.
\end{equation}
(One way to see this is to note that on every nilpotent Lie group there is a canonical homogeneous norm controlling the Euclidean norm.) 

\begin{remark}
This is essentially the first time in this paper that we make use of the explicit structure of the underlying space. The framework of \S3 applies to more general spaces (than the nilpotent Lie groups) and indeed the $U$-bounds of \cite{HeZe09} were originally proved for a general sub-gradient $\nabla$ on $\bbR^n$ and in fact the ideas extended to the manifold setting as well. That said, the setting of nilpotent Lie groups provides a rich family of examples to study and will still remain the principal focus of this paper. 
\end{remark}

Fix a second constructive parameter, this time a spin configuration $y \in \Omega$, and denote by $m_y = \frac{1}{\abs{B(R)}} \int_{B(R)} f(y_i)dy_i \vcentcolon= \frac{1}{\abs{B(R)}} \int_{B(R)} f_i^y(y_i)dy_i$, that is the average of $f$ over $B(R) \sbs \bbG^i$ with boundary conditions $y$. We study the LHS of \bf{(A2)} by applying the cutoff
\begin{equation}\label{geodesicI1I2}
\begin{aligned}
\nu\bbE_i^\omega \abs{f - \bbE_i^\omega f}^q &\leq 2^q \nu\bbE_i^\omega \abs{f - m_y}^q \\
&= 2^q\nu\bbE_i^\omega\abs{f - m_y}^q\mathbbm{1}_{\{\eta_i + \bar{\eta}_i \leq L_1\}} + 2^q\nu\bbE_i^\omega\abs{f - m_y}^q\mathbbm{1}_{\{\eta_i + \bar{\eta}_i > L_1\}}
\end{aligned}
\end{equation}
where $\tilde{\eta}_i \vcentcolon= \sum_{j \sim i} \eta_j$, so that the function $\eta_i + \bar{\eta}_i$ is a function of $x_i$ and $x_j$ for $j \sim i$. We first study $\smash{I_1 = \abs{f - m_y}^q\mathbbm{1}_{\{\eta_i + \bar{\eta}_i \leq L_1\}}}$. Since $\eta$ is nonnegative, we know $\{\eta_i + \bar{\eta}_i \leq L_1\} \sbs \{\eta_i, \bar{\eta}_i \leq L_1\}$ and therefore, since $\phi$ and $V$ are continuous, we can bound the potential in this region by 
\begin{equation}\label{I1Uibound}
\abs{U_i^\omega} \leq \sup_{\eta(h) \leq L_2} \abs{\phi(h)} + 2D\beta\sup_{\{\eta(h_1), \eta(h_2) \leq L_2\}} \abs{V(h_1, h_2)} \eqqcolon M.
\end{equation}
In fact, taking $\beta \ll 1$, we can make $M$ independent of $\beta$ and solely dependent on $L_2$ which in turn depends on $L_1$. Denote by $S(L_1) = \{\bar{\eta}_i \leq L_1\}$. We will remove the dependence of $I_1$ as a function of $x \in \Omega$ by integrating $I_1$ against $\bbE_i^\omega$; in doing so, we ``freeze'' the boundary conditions according to $\omega$ which will allow us to take $\smash{\mathbbm{1}_{\{\eta_i + \bar{\eta}_i \leq L_1\}} \leq \mathbbm{1}_{\{\bar{\eta}_i \leq L\}} = \mathbbm{1}_{S(L)}}$ outside the integral in $x_i$. In particular, $\bbE_i^\omega I_1 = 0$ whenever $\omega \notin S(L_1)$. Then  
\begin{equation}\label{eiwI1}
\begin{aligned}
\bbE_i^\omega I_1 &\leq \int_{B(R)} \abs{f(x_i) - \frac{1}{\abs{B(R)}}\int_{B(R)} f(y_i)dy_i}^q \frac{e^{-U_i^\omega}}{Z_i^\omega} dx_i \mathbbm{1}_{S(L_1)} \\
&\leq \frac{e^{M}}{Z_i^\omega} \int_{B(R)} \abs{\frac{1}{B(R)} \int_{B(R)} f(x_i) - f(y_i)dy_i}^q dx_i \mathbbm{1}_{S(L_1)} 
%%%%%%%
\vphantom{\int_{B(R)} \abs{f(x_i) - \frac{1}{\abs{B(R)}}\int_{B(R)} f(y_i)dy_i}^q \frac{e^{-U_i^\omega}}{Z_i^\omega} dx_i \mathbbm{1}_{S(L_1)}} \\
%%%%%%%
&\leq \frac{e^{M}}{Z_i^\omega \abs{B(R)}} \int\int \abs{f(x_i) - f(y_i)}^q \mathbbm{1}_{B(R)}(x_i)\mathbbm{1}_{B(R)}(y_i)dy_idx_i \mathbbm{1}_{S(L_1)}
%%%%%%%
\vphantom{\int_{B(R)} \abs{f(x_i) - \frac{1}{\abs{B(R)}}\int_{B(R)} f(y_i)dy_i}^q \frac{e^{-U_i^\omega}}{Z_i^\omega} dx_i \mathbbm{1}_{S(L_1)}} \\
%%%%%%%
&\leq \frac{e^{M}}{Z_i^\omega \abs{B(R)}} \int\int \abs{f(x_i) - f(x_i \circ y_i)}^q \mathbbm{1}_{B(R)}(x_i)\mathbbm{1}_{B(R)}(x_i \circ y_i)dy_idx_i \mathbbm{1}_{S(L_1)}
%%%%%%%
\vphantom{\int_{B(R)} \abs{f(x_i) - \frac{1}{\abs{B(R)}}\int_{B(R)} f(y_i)dy_i}^q \frac{e^{-U_i^\omega}}{Z_i^\omega} dx_i \mathbbm{1}_{S(L_1)}}
%%%%%%%
\end{aligned}
\end{equation}
where the third inequality follows from an application of Jensen's inequality (with respect to the uniform measure on $B(R)$) and the last inequality follows from the change of variable $y_i \mapsto x_i \circ y_i$. Note any change of variable in $x_i$ does not affect $\smash{\mathbbm{1}_{S(L_1)}}$ since membership in $S(L_1)$ does not depend on $x_i$. 

Let $\gamma: [0, d(y_i)] \rightarrow \bbG$ be a geodesic from $0$ to $y_i$ such that $\abs{\dot{\gamma}(t)} \leq 1$ for all $t$. Then 
\begin{equation}\label{geodesicIneq0}
\begin{aligned}
\abs{f(x_i) - f(x_i \circ y_i)}^q &= \abs{\int_0^{d(y_i)} \frac{d}{dt}f(x_i \circ \gamma(t))}^q \\
&= \abs{\int_0^{d(y_i)} \nabla_i f(x_i \circ \gamma(t)) \cdot \dot{\gamma}(t)}^q \\
&\leq d^{q/p}(y_i) \int_0^{d(y_i)} \abs{\nabla_i f(x_i \circ \gamma(t))}^q dt \vphantom{\abs{\int_0^{d(y_i)} \nabla_i f(x_i \circ \gamma(t)) \cdot \dot{\gamma}(t)}^q} 
\end{aligned}
\end{equation}
by H\"older's inequality.
% was Jensen, should be correct now
Since $x_i \in B(R)$ and $x_i \circ y_i \in B(R)$, the estimates
\begin{equation}\label{geodesicIneq1}
d(y_i) \leq d(x_i) + d(x_i \circ y) \leq 2R
\end{equation}
and 
\begin{equation}\label{geodesicIneq2}
d(x_i \circ \gamma(t)) \leq d(x_i) + d(\gamma(t)) \leq d(x_i) + d(y_i) \leq 3R 
\end{equation}
hold for all $t \in [0, d(y_i)]$, since $d$ is a metric and $\gamma$ is a geodesic. If we insert \eqref{geodesicIneq0}, \eqref{geodesicIneq1}, and \eqref{geodesicIneq2} into \eqref{eiwI1}, we find 
\begin{align*}
\bbE_i^\omega I_1 &\leq \frac{2Re^{M}}{Z_i^\omega \abs{B(R)}} \int\int\int_0^{d(y_i)} \abs{\nabla_i f(x_i \circ \gamma(t))}^q \mathbbm{1}_{B(2R)}(y_i) \mathbbm{1}_{B(3R)}(x_i \circ \gamma(t)) dt dy_i dx_i \mathbbm{1}_{S(L_1)} \\
&\leq \frac{2Re^{M}}{Z_i^\omega \abs{B(R)}} \int\int\int_0^{d(y_i)} \abs{\nabla_i f(x_i)}^q \mathbbm{1}_{B(2R)}(y_i) \mathbbm{1}_{B(3R)}(x_i) dt dy_i dx_i \mathbbm{1}_{S(L_1)} 
%%%%%%%
\vphantom{\frac{2L_1e^{M}}{Z_i^\omega \abs{B(R)}} \int\int\int_0^{d(y_i)} \abs{\nabla_i f(x_i \circ \gamma(t))}^q \mathbbm{1}_{B(2R)}(y_i) \mathbbm{1}_{B(3R)}(x_i \circ \gamma(t)) dt dy_i dx_i \mathbbm{1}_{S(L_1)}} \\
%%%%%%%
&\leq \frac{4R^2e^{M}}{Z_i^\omega \abs{B(R)}} \int\int \abs{\nabla_i f(x_i)}^q \mathbbm{1}_{B(2R)}(y_i) \mathbbm{1}_{B(3R)}(x_i) dy_i dx_i \mathbbm{1}_{S(L_1)}
%%%%%%%
\vphantom{\frac{2L_1e^{M}}{Z_i^\omega \abs{B(R)}} \int\int\int_0^{d(y_i)} \abs{\nabla_i f(x_i \circ \gamma(t))}^q \mathbbm{1}_{B(2R)}(y_i) \mathbbm{1}_{B(3R)}(x_i \circ \gamma(t)) dt dy_i dx_i \mathbbm{1}_{S(L_1)}} \\
%%%%%%%
&\leq \frac{4R^2 e^{M} \abs{B(2R)}}{Z_i^\omega \abs{B(R)}} \int \abs{\nabla_i f(x_i)}^q \mathbbm{1}_{B(3R)}(x_i) dx_i \mathbbm{1}_{S(L_1)}
%%%%%%%
\vphantom{\frac{2L_1e^{M}}{Z_i^\omega \abs{B(R)}} \int\int\int_0^{d(y_i)} \abs{\nabla_i f(x_i \circ \gamma(t))}^q \mathbbm{1}_{B(2R)}(y_i) \mathbbm{1}_{B(3R)}(x_i \circ \gamma(t)) dt dy_i dx_i \mathbbm{1}_{S(L_1)}} \\
%%%%%%%
&\leq 4R^2 e^{2M} \frac{\abs{B(2R)}}{\abs{B(R)}} \int \abs{\nabla_i f(x_i)}^q \frac{e^{-U_i^\omega}}{Z_i^\omega} dx_i \mathbbm{1}_{S(L_1)}
%%%%%%%
\vphantom{\frac{2L_1e^{M}}{Z_i^\omega \abs{B(R)}} \int\int\int_0^{d(y_i)} \abs{\nabla_i f(x_i \circ \gamma(t))}^q \mathbbm{1}_{B(2R)}(y_i) \mathbbm{1}_{B(3R)}(x_i \circ \gamma(t)) dt dy_i dx_i \mathbbm{1}_{S(L_1)}} \\
%%%%%%%
&\leq 4R^2 e^{2M} \frac{\abs{B(2R)}}{\abs{B(R)}} \bbE_i^\omega \abs{\nabla_i f}^q \mathbbm{1}_{S(L_1)},
%%%%%%%
\vphantom{\frac{2L_1e^{M}}{Z_i^\omega \abs{B(R)}} \int\int\int_0^{d(y_i)} \abs{\nabla_i f(x_i \circ \gamma(t))}^q \mathbbm{1}_{B(2R)}(y_i) \mathbbm{1}_{B(3R)}(x_i \circ \gamma(t)) dt dy_i dx_i \mathbbm{1}_{S(L_1)}}
%%%%%%%
\end{align*}
where in the first inequality we use $d(y_i)^{q/p} \leq 2R$, in the second the change of variable $x_i \mapsto x_i \circ \gamma(t)$, and in the fifth inequality the bound \eqref{I1Uibound}, giving $1 \leq e^{M_{L_1}}/e^{-U_i^\omega}$. 

Let us recall $R$ and $M_{L_1}$ depend on our choice of $L_1$ and are independent of $\omega$. It turns out also $\frac{\abs{B(2R)}}{\abs{B(R)}}$ is bounded above independently of $R$; indeed on a nilpotent Lie group $\bbG$ the balls associated to the Carnot-Carath\'eodory distance $d$ behave much like Euclidean balls in that $\abs{B(R)} \simeq R^Q$ for an appropriate power $Q$, called the homogeneous dimension, and so $\abs{B(2R)}/\abs{B(R)}$ is a constant. That said, such a bound is not needed for what remains of the proof. With $C = C(L_1) = 4R^2e^{2M}\frac{\abs{B(2R)}}{\abs{B(R)}}$, applying the DLR equation we obtain 
\begin{equation}\label{I1H2SGI}
\nu I_1 \leq C\nu(\abs{\nabla_if}^q \mathbbm{1}_{S(L)}) \leq C\nu\abs{\nabla_i f}^q. 
\end{equation}

As mentioned earlier, the estimate for $\smash{I_2 = \abs{f - m_y}^q\mathbbm{1}_{\{\eta_i + \bar{\eta}_i > L_1\}}}$ is easier to obtain: the $U$-bound \eqref{dUbound} gives
\begin{equation}\label{I2H2SGI}
\begin{aligned}
\nu I_2 &\leq \frac{1}{L_1}\nu\left(\abs{f - m_y}^q(\eta_i + \bar{\eta}_i)\right) 
%%%%%%%
\vphantom{\frac{K_2}{L_1} \left(\nu\abs{f - m_y}^q + \sum_{k \in \bbZ^D} \sum_{j \sim i} \left(c_2^{\norm{k - i}} + c_2^{\norm{k - j}}\right) \nu\abs{\nabla_kf}^q\right)} \\
%%%%%%%
&\leq \frac{K_2}{L_1} \left(\nu\abs{f - m_y}^q + \sum_{k \in \bbZ^D} \sum_{\norm{j - i} \leq 1} c_2^{\norm{k - j}} \nu\abs{\nabla_k f}^q\right) 
%%%%%%%
\vphantom{\frac{K_2}{L_1} \left(\nu\abs{f - m_y}^q + \sum_{k \in \bbZ^D} \sum_{j \sim i} \left(c_2^{\norm{k - i}} + c_2^{\norm{k - j}}\right) \nu\abs{\nabla_kf}^q\right)} \\
%%%%%%%
&\leq \frac{K_2}{L_1} \left(\nu\abs{f - m_y}^q + \sum_{k \in \bbZ^D} \sum_{j \sim i} \left(c_2^{\norm{k - i}} + c_2^{\norm{k - j}}\right) \nu\abs{\nabla_kf}^q\right) \\
&\leq \frac{4dK_2}{L_1} \left(\nu\abs{f - m_y}^q + \sum_{k \in \bbZ^D} c_2^{\dist(k, i_1)} \nu\abs{\nabla_kf}^q\right) 
%%%%%%%
\vphantom{\frac{K_2}{L_1} \left(\nu\abs{f - m_y}^q + \sum_{k \in \bbZ^D} \sum_{j \sim i} \left(c_2^{\norm{k - i}} + c_2^{\norm{k - j}}\right) \nu\abs{\nabla_kf}^q\right)}
%%%%%%%
\end{aligned}
\end{equation}
with $i_1 = \{i\} \cup \{\sim i\}$ as before. Note here we use that $m_y$ is a fixed real number so $\nabla_j m_y = 0$ for all $j \in \bbZ^D$. It follows since $K_2 = \cO(1)$ we can take $L_1$ large enough (but constant in $\beta$) to ensure $4dK_2/L \leq \frac{1}{2}$. We then obtain that $C = C(L_1)$ is also $\cO(1)$ and putting these estimates together we conclude
\begin{align*}
\nu\abs{f - m_y}^q = \nu I_1 + \nu I_2 &\leq C\nu\abs{\nabla_if}^q + \frac{1}{2}\nu\abs{f - m_y}^q + \frac{1}{2}\sum_{k \in \bbZ^D} c_2^{\dist(k, i_1)}\nu\abs{\nabla_kf}^q
\end{align*}
so that from \eqref{geodesicI1I2} one has 
\[ \nu\bbE_i^\omega\abs{f - \bbE_i^\omega f}^q \leq 4C\nu\abs{\nabla_if}^q + 2\sum_{k \in \bbZ^D} c_2^{\dist(k, i_1)}\nu\abs{\nabla_kf}^q. \]
Thus we see that $A = \max(B(1 + 2K_2), 4C + 2) = \cO(1)$ satisfies \bf{(B1)}, which finishes the proof. 
\end{proof}

\section{Examples}
%Section 5
In this section, we construct examples of potentials and interactions, which may depend on Carnot-Carath\'eodory distance $d$ or a smooth homogeneous norm, satisfying \bf{(C1)} and \bf{(C2)}. Consequently, this will provide new examples where a $q$-SGI holds. To this end, the usual strategy is to find a phase $\phi$ and a function $\eta$ satisfying the $U$-bound of \bf{(C1)}, and then equipping this with an interaction $V$ satisfying \bf{(C2)}. Before we proceed, we first give a brief introduction into the Heisenberg group, which is the prototypical example of a nontrivial nilpotent Lie group and perhaps one of the simplest toy models in the theory of subelliptic (or subriemannian) geometry.

\subsection{The Heisenberg group}
In this paper, the Heisenberg group $\bbH$ is identified as the Lie group $(\bbR^3, \circ)$ equipped with group law
\begin{equation}\label{def:grouplaw}
\left(\begin{array}{c} x_1 \\ x_2 \\ x_3 \\ \end{array}\right) \circ \left(\begin{array}{c} y_1 \\ y_2 \\ y_3 \\ \end{array}\right) = \left(\begin{array}{c} x_1 + y_1 \\ x_2 + y_2 \\ x_3 + y_3 + 2(x_1y_2 - x_2y_1) \\ \end{array}\right).
\end{equation}
A property that distinguishes $\bbH$ from the trivial group $(\bbR^3, +)$ is the fact its Lie algebra $\mf{h}$ of left-invariant vector fields is generated by the two vector fields 
\begin{equation}\label{def:vecfields}
\begin{aligned}
X^1 &= \partial_1 + 2x_2\partial_3, \\
X^2 &= \partial_2 - 2x_1\partial_3,
\end{aligned}
\end{equation}
in the sense $X^1$ and $X^2$, together with their commutator $X^3 \vcentcolon= [X^1, X^2] = -4\partial_3$, span all of $\mf{h}$. The importance of the two vector fields $X^1$ and $X^2$ is that they allow us to define a differential calculus of sorts in which the only differential operators that exist are $X^1$ and $X^2$, and in particular they allow us to define analogues of objects in $\bbR^3$ that take into account the group structure of $\bbH$. 
% see pp20 of BLU07	

For instance, we may define a metric on $\bbH$ by considering curves whose derivative belongs to the sub-bundle spanned by $X^1$ and $X^2$. A Lipschitz curve $\gamma: [0, T] \rightarrow \bbH$ is said to be \emph{subunit} with respect to $X^1$ and $X^2$ if there exist measurable coefficients $c(t) = (c_1(t), c_2(t))_{t \in [0, T]}$ such that $\dot{\gamma}(t) = c_1(t)X^1(\gamma(t)) + c_2(t)X^2(\gamma(t))$ and $c_1^2(t) + c_2^2(t) \leq 1$ for almost all $t \in [0, T]$. It turns out on $\bbH$ we can always connect two points $x, y \in \bbH$ by a subunit curve by Chow's theorem (see \cite[Chapter~19]{BLU07}) and so we may define 
\begin{equation}\label{def:ccdist}
d(x, y) = \min \left\{T > 0 \mid \exists \text{ a subunit curve } \gamma: [0, T] \rightarrow \bbH \Mwith \begin{array}{c} \gamma(0) = x \\ \gamma(T) = y\end{array}\right\}.
\end{equation}
% see footnote pp70 of BLU07
From this perspective the distance between two points is the minimum amount of time required to connect two points with a curve with bounded coefficients and which can only ``travel in the directions'' of $X^1$ and $X^2$. Alternatively, we may remove the condition of bounded coefficients and instead, with the length of a curve defined as $\ell(\gamma) = \int_0^T \abs{c(t)}^2dt$, define $d$ as the minimum length of a curve which connects two points in unit time. In any case the two definitions are equivalent and $d$ is a metric \cite[Proposition~3.1]{JeSa87}, called the \emph{Carnot-Carath\'eodory metric}, and it endows $\bbH$ with geodesics that are distinguished from the geodesics of $\bbR^3$. More precisely, the geodesics of $\bbR^3$ are straight lines, but the geodesics of $\bbH$ with respect to $d$ are helices \cite[Theorem~2.1]{HaZi15}. 
% prop3.1 JeSa87 also appears as prop5.2.3 in BLU07

Of greater importance to this paper is the fact $X^1$ and $X^2$ allow us to define an analogue of the Euclidean gradient $\nabla = (\partial_1, \partial_2, \partial_3)$, namely the \emph{sub-gradient} 
\begin{equation}\label{def:subgradient}
\nabla_\bbH = (X^1, X^2)
\end{equation}
on $\bbH$. Similarly, we may define an analogue of the euclidean Laplacian $\Delta = \partial_1^2 + \partial_2^2 + \partial_3^2$, called the \emph{sub-Laplacian} as follows
\[ \cL_\bbH = (X^1)^2 + (X^2)^2. \]
The reader may notice the definition of the sub-gradient and the sub-Laplacian depend on the choice of generators $X^1$ and $X^2$ of $\mf{h}$, but there is a canonical choice (and thus we may refer to ``the'' sub-gradient and ``the'' sub-Laplacian) of generators (which is taken in \eqref{def:vecfields}) associated to the group law which makes the Jacobian of the left translation operator at zero coincide with the identity matrix, in analogue with the Euclidean case. The sub-Laplacian will not feature prominently in this paper but we refer the reader to \cite[Chapter~5]{BLU07} for more details. 

The metric $d$ also induces a ``norm'' on $\bbH$, by defining with some abuse of notation, 
\begin{equation}\label{def:ccnorm}
d(x) \vcentcolon= d(x, 0).
\end{equation}
The fact $d$ is a norm requires some qualification since a'priori $d$ need not satisfy the triangle inequality nor be positively homogeneous. If these properties are given the expected description which accounts for the group structure on $\bbH$, then $d$ \emph{does} satisfy the triangle inequality defined with the group law $\circ$ in \eqref{def:grouplaw} instead of the addition law $+$ on $\bbR^3$, i.e. 
\[ d(x \circ y) \leq d(x) + d(y). \]
Similarly, we may equip $\bbH$ with a family of dilations $(\delta_\lambda)_{\lambda > 0}$ mirroring the stratification of its Lie algebra:
\begin{equation}\label{def:dilation}
\delta_\lambda(x_1, x_2, x_3) = (\lambda x_1, \lambda x_2, \lambda^2 x_3).
\end{equation}
In doing so, we realise $\bbH$ as a homogeneous Carnot group and we usually write the points of $\bbH$ as $x = (w, z) \in \bbR^2 \times \bbR$ so that \eqref{def:dilation} reads $\delta_\lambda(w, z) = (\lambda w, \lambda^2z)$. With respect to these dilations, $d$ \emph{is} positively homogeneous, i.e. $d(\delta_\lambda(x)) = \lambda d(x)$, and for this reason we call $d$ a homogeneous norm. 

It turns out that $d$ shares a rather important property with the Euclidean norm $x \mapsto \abs{x}$: it satisfies the eikonal equation away from the centre $\{(0, z) \mid z \in \bbR\} \sbs \bbH$. 
\begin{proposition}[{\cite[Theorem~3.8]{Mon00}}]\label{prop:eikonal}
The metric $d$ satisifes
\[ \abs{\nabla_\bbH d(w, z)} = 1 \]
for all $w \neq 0$. 
\end{proposition}
In light of this result we might be led to believe that $d$ is the most ``natural'' choice of homogeneous norm on $\bbH$ in that it analogises the Euclidean norm. However it does depart from the Euclidean norm; for instance, the isoperimetric problem in $\bbH$ is \emph{not} optimised by $d$-balls \cite[Proposition~4.16]{Mon00}. In this paper, we study also another homogeneous norm $N$, called the Kaplan norm. On a general homogeneous Carnot group, the Kaplan norm is expressed as the fundamental solution of the sub-Laplacian raised to a certain exponent, and in the specific case of $\bbH$ it takes the explicit form
\begin{equation}\label{def:kaplan}
N(w, z) = (\abs{w}^4 + 16\abs{z}^2)^{1/4}.
\end{equation}
Interestingly enough homogeneous norms are topologically equivalent, and yet may exhibit different analytic behaviour. For instance %by Proposition \ref{prop:eikonal}, 
$d$ is smooth away from the centre of $\bbH$, while $N$ is smooth everywhere except at zero. Moreover,  $d$ satisfies the eikonal equation, whereas for $N$ one has 
\[\smash{\lvert\nabla_\bbH N(w, z)\rvert = \frac{\abs{w}}{N}},\]
and in particular $\smash{\lvert\nabla_\bbH N\rvert}$ can be arbitrarily small far away from the origin. The proofs of these results on homogeneous norms can be found e.g. in \cite[Chapter~5]{BLU07}.

\subsection{The Carnot-Carath\'eodory distance $d$}
The problem of finding $\phi$ and $\eta$ satisfying \bf{(C1)} on the Heisenberg group has been studied in the literature. As mentioned earlier, one already has the $q$-logarithmic Sobolev inequality courtesy of \cite[Theorem~5.2]{InPa09}, which made use of the following $U$-bound for the Carnot-Carath\'eodory distance $d$ first appearing in \cite[Theorem~2.3]{HeZe09}:
\begin{equation}\label{uboundford}
\int \abs{g}^qd^{q(p-1)}e^{-\alpha d^p}dh \lesssim \int \left(\abs{\nabla g}^q + \abs{g}^q\right)e^{-\alpha d^p}dh
\end{equation}
for $\alpha > 0$, $p > 1$, and $q \geq 1$. If we consider the interactions $V$ for which 
\[ \abs{\nabla_iV_{ij}(h_i, h_j)}^q + \abs{\nabla_jV_{ij}(h_i, h_j)}^q \leq B(1 + d^{q(p-1)}(h_i) + d^{q(p-1)}(h_j)), \]
one starting point is $V_{ij} = P(d_i, d_j)$ for a polynomial $P$ and where $V_{ij} \vcentcolon= V(h_i, h_j)$, $d_i \vcentcolon= d(h_i)$, and $d_j \vcentcolon= d(h_j)$, in analogue with the shorthands $\eta_i$ and $\eta_j$ in the previous section. The subgradient is 
\[ \abs{\nabla_iV_{ij}} = \abs{\nabla_i d_i}\abs{\partial_1P(d_i, d_j)} \leq \abs{\partial_1P(d_i, d_j)} \]
by Proposition \ref{prop:eikonal}, and assuming $P$ has degree $r \in \bbN$ its subgradient is controlled by a polynomial $\partial_1P$ of degree at most $r-1$. The same goes for $\abs{\nabla_jV_{ij}}$, which means $\abs{\nabla_iV_{ij}(h_i, h_j)}^q + \abs{\nabla_jV_{ij}(h_i, h_j)}^q$ is controlled by $1 + d_i^{q(r-1)} + d_j^{q(r-1)}$ modulo a constant. It follows \bf{(C2)} amounts to $q(r - 1) \leq q(p - 1)$, i.e. $r \leq p$. 

\begin{theorem}\label{sgiford0}
%Section 5 Thm5_1%
Let $\alpha > 0$, $p > 1$, and $q \in (1, 2]$ be dual to $p$. Define a local specification on $\bbH^{\bbZ^D}$ by the following potential:
\[ U_\Lambda^\omega(x_\Lambda) = \alpha \sum_{i \in \Lambda} d^p(x_i) + \sum_{i, j \in \Lambda} \beta_{ij} P(d(x_i), d(x_j)) + \sum_{i \in \Lambda, \, j \notin \Lambda} \beta_{ij} P(d(x_i), d(\omega_j)) \]
with $\abs{\beta_{i,j}} < \beta$ for all $i, j \in \bbZ^D$, and for $P$ a polynomial of degree $\bbN \ni r \leq p$. Then for all sufficiently small $\beta > 0$, there exists a unique Gibbs measure $\nu$ which satisfies the $q$-spectral gap inequality
\[ \nu\abs{f - \nu f}^q \leq C_{SG}\nu\abs{\nabla f}^q. \]
\end{theorem}

Let us now discuss some natural generalisations. The first pertains to whether the phase $\phi = \alpha d^p$ can be given some lower order terms, that is if $\phi$ can be a semibounded generalised polynomial of order $p > 1$. The answer is yes. 

\begin{proposition}
%Prop 5.1
The conclusion of Theorem \ref{sgiford0} holds for $\phi = \alpha d^p + \tilde{P}(d)$ where $\tilde{P}$ is a generalised polynomial comprised of monomials of the form $d^\alpha$ for $0 \leq \alpha < p$. 
\end{proposition}

\begin{proof}
Much of this follows from the perturbation theory of \cite[Theorem~2.2]{HeZe09} which implies \eqref{uboundford} is valid for the potential $U = \alpha d^p + W$ where $W$ is any potential whose subgradient is controlled by $d^{p-1}$; in particular this allows $W$ to have lower order terms of the form $d^{p-\epsilon}$ for any $0 < \epsilon \leq p$. There is an issue with fractional exponents, e.g. $d^\epsilon$ with $0 < \epsilon < 1$, since $d^{\epsilon-1}$ is not controlled by $d^\gamma$ for any positive $\gamma > 0$ in a neighbourhood of the origin. 

Nonetheless such singularities are inconsequential; the purpose of a $U$-bound is control over what happens at infinity. To allow for $W$ having singularities near the origin, we can follow the proof of \cite[Theorem~2.1]{HeZe09}: decompose $f$ into a part localised at the origin and another at infinity. The former is controlled by $\int f^2d\mu$ modulo a constant, the latter by the arguments of \cite[Theorem~2.2]{HeZe09} since we have cut away from the singularities near the origin. 
\end{proof}

The second generalisation pertains to the interaction. One may hope that $V$ may be replaced with a generalised polynomial. Indeed, the arguments for Theorem \ref{sgiford0} did not rely on the integrality of the powers. Note that at this stage it does not appear possible to relax \bf{(C2)} to allow for singularities in the subgradients $\nabla_i V_{ij}$, $\nabla_j V_{ij}$, in the same way we allowed for singularities in the phase perturbation $\tilde{P}$. The problem lies with the fact the interaction in the statement of Theorem \ref{thm:singletoglobal} is the $U$-bound function, and we shall sidestep the unfortunate obstacle by assuming the powers are either zero or at least one. 

\begin{comment}
\begin{remark}
\textcolor{red}{Added remark for why $p$ is either zero or $\geq 1$.}
\end{remark}
\end{comment}

\begin{proposition}
%Prop 5.2
The conclusion of Theorem \ref{sgiford0} holds for $V_{ij} = P(d_i, d_j)$ where $P$ is a generalised polynomial comprised of monomials of the form $d_i^\alpha d_j^\beta$ for $\alpha, \beta \in \{0\} \cup [1, \infty)$ and $0 \leq \alpha + \beta \leq p$. 
\end{proposition}

The third generalisation concerns again the interaction, namely the question of whether other homogeneous norms can enter. Indeed, this boils down to a simple verification of \bf{(C2)} and with enough conditions we can answer affirmatively. 

\begin{proposition}
%Prop 5.3
The conclusion of Theorem \ref{sgiford0} holds for $V_{ij} = Q(N^1_i, N^1_j, \cdots, N^k_i, N^k_j)$ for $N^1, \cdots, N^k$ a collection of homogeneous norms on $\bbH$ smooth almost everywhere with bounded subgradient $\abs{\nabla N^\ell} \leq K_\ell$ for each $\ell = 1, \cdots, k$ and $Q$ a generalised polynomial comprised of monomials of the form $\prod_{\ell=1}^k (N^\ell_i)^{\alpha_\ell}(N^\ell_j)^{\beta_\ell}$ for $\alpha_\ell, \beta_\ell \in \{0\} \cup [1, \infty)$ and $0 \leq \sum_{\ell=1}^k \alpha_\ell + \beta_\ell \leq p$. In particular it holds for $k = 1$ and $N^1 = d$. 
\end{proposition}

Aside from $d$ we will study in the next section the Kaplan norm $N$ which (along with $d$) is a standard example of a homogeneous norm smooth almost everywhere (in fact, away from the origin) with bounded subgradient. One can also construct from two homogeneous norms another by taking, for instance, linear combinations or geometric means (or square roots of squares, say). In fact we may also work with quasinorms, that is norms which are not positive definite and which may vanish away from the origin. The prototypical example of a quasinorm is the horizontal norm $\abs{h}$ which is introduced in the next section and that plays the role of a ``degenerated Euclidean norm" acting only on a proper subset of the coordinates in $\bbH$. 

Note if a homogeneous norm is smooth away from the origin then it automatically has bounded subgradient by homogeneity considerations. However, homogeneous norms smooth almost everywhere in general need not have unbounded subgradient; one such example is the geometric mean $N^{1/2}\abs{h}^{1/2}$ of $N$ and $\abs{h}$ whose subgradient blows up along the hyperplane $\{ \abs{h} = 0 \}$. This is not a norm, but can be made into a norm, for instance by adding a small norm $\epsilon N$. Note as well a homogeneous norm need not even be smooth anywhere (let alone almost everywhere). For instance, take a smooth homogeneous norm and deform its ball nonsmoothly while preserving homogeneity and antipodal points.

\begin{remark}
Note by \cite[Proposition~5.16.4]{BLU07} that a smooth homogeneous norm is Lipschitz with respect to itself and therefore Lipschitz with respect to the Carnot-Carath\'eodory distance. By \cite[Theorem~2.2.1]{Mon01} such a norm has bounded subgradient. 
\end{remark}

\begin{proof}
By the Leibniz rule 
\[ \abs{\nabla_i \prod_{\ell=1}^k (N^\ell_i)^{\alpha_\ell}(N^\ell_j)^{\beta_\ell}} \leq \sum_{\ell=1}^k \alpha_\ell (N^\ell_i)^{\alpha_\ell-1} (N^\ell_j)^{\beta_\ell} \abs{\nabla_i N^\ell_i} \leq \sum_{\ell=1}^k \alpha_\ell K_\ell (N^\ell_i)^{\alpha_\ell-1} (N^\ell_j)^{\beta_\ell}. \]
By equivalence of homogeneous norms we can control each $N^\ell_i, N^\ell_j$ with $d_i, d_j$ respectively and it now reduces to the case $k = 1$ and $N^k = N^1 = d$. 
\end{proof}

\begin{remark}
The monomials $\prod_{\ell=1}^k (N^\ell_i)^{\alpha_\ell}(N^\ell_j)^{\beta_\ell}$ can be realised in terms of powers of two homogeneous norms $\varrho^\ell, \delta^\ell$; indeed, let us suppose that $(N^\ell)$ is a family of homogeneous norms then $(\prod_\ell (N^\ell)^{\alpha_\ell})^{1/\sum_\ell \alpha_\ell}$ is itself a homogeneous norm. So we may restate the previous proposition with $Q$ replaced by $Q_{ij} = \sum_{\ell=1}^k (\varrho^\ell_i)^{\alpha_\ell} (\delta^\ell_j)^{\beta_\ell}$, i.e. $Q$ is a sum of products of powers of homogeneous norms. 
\end{remark}

\begin{remark}
We may also be interested in interactions of the form $V_{ij}(x_i, x_j) = d(x_j^{-1} \circ x_i)$, that is replacing the Euclidean ``arithmetic subtraction" $d(x_i) - d(x_j)$ with the group subtraction $d(x_j^{-1} \circ x_i)$. In this case the inversion must appear on the correct side; by \cite[Proposition~5.2.7]{BLU07} one finds 
\[ \abs{\nabla_a d(b^{-1} \circ a)} \leq 1, \quad \abs{\nabla_b d(b^{-1} \circ a)} \leq 1 \]
since $d$ commutes with left translation so $\nabla_ad(b^{-1} \circ a) = (\nabla_ad)(\tau_{b^{-1}}(a))$ while the $\nabla_b$ bound is exactly the same since $d(h) = d(h^{-1})$. One however can show for $N$ the Kaplan norm that the subgradient in the $x_j$-coordinate of $N(x_i \circ x_j^{-1})$ is unbounded, for instance in a neighbourhood of $(x, y, z) = (1, 1, z)$. In any case by the pseudo-triangle inequality, powers of $d(x_j^{-1} \circ x_i)$ can appear in the interaction.
\end{remark}

\begin{remark} We remark that in the recent preprint \cite{KP19} the authors study some models indicated above by partially alternative techniques, focussing on interactions dependent on the Carnot-Carath\'eodory distance to achieve also the logarithmic Sobolev inequality.
See also some older references \cite{Pap10, Pap14, Pap18} by I. Papageorgiou.
\end{remark}

\subsection{The Kaplan norm $N$}
One serious deficiency in the case of the Kaplan norm $N$ is that at the time the $q$-SGI was proven for the measure $Z^{-1}e^{-\alpha N^p}$ in \cite[Theorem~4.5.5]{Ing10} for $\alpha > 0$ and $p \geq 2$, the arguments used a degenerate $U$-bound $\eta(h) = N^{p-2}\abs{h}^2$ where $\abs{h} = \abs{(x, y, z)} = (x^2 + y^2)^{1/2}$ is the horizontal norm, in the sense $\eta$ fails to diverge to infinity in all directions, in particular along the $z$-axis where $\abs{h} = x = y = 0$. But recently a $U$-bound for the Kaplan norm more suitable for our applications was derived, which we now state.

\begin{theorem}[{\cite[Theorem~1]{DaZe21a}}]\label{thm:esther}
The following $U$-bound
\[ \int \abs{g}^q N^{p-3} e^{-\alpha N^p} dh \lesssim \int \left(\abs{\nabla g}^q + \abs{g}^q\right)e^{-\alpha N^p}dh \]
holds for any $\alpha > 0$, $p, q \geq 2$, and $g$ supported away from the unit $N$-ball $\{N < 1\}$.  
\end{theorem}

\begin{remark}
By the perturbation theory of \cite[Theorem~2.2]{HeZe09} the result can be extended to any semibounded polynomial of order $p \geq 2$. Note of course $N^{p-3}$ does not blow up at infinity unless $p > 3$. 
\end{remark}

If $p > 3$, then $\eta = N^{p-3}$ diverges to infinity in all directions as required, and the $U$-bound can be extended to all of $\bbH$ as $N^{p-3} \leq 1$ in $\{N < 1\}$, meaning we obtain the $U$-bound
\begin{equation}\label{uboundforn}
\int \abs{g}^qN^{p-3}e^{-\alpha N^p}dh \lesssim \int \left(\abs{\nabla g}^q + \abs{g}^q\right)e^{-\alpha N^p}dh
\end{equation}
for $\alpha > 0$, $p > 3$, and $q \geq 2$. Note as before it suffices to consider $q = 2$. By the same arguments for the Carnot-Carath\'eodory distance the $U$-bound for $\phi = \alpha N^p$ and $\eta = N^{p-3}$ furnishes a $2$-SGI for the Gibbs measure $\nu$. Note the $q$-SGI of \cite[Theorem~4.5.5]{Ing10} in the $1$-dimensional setting holds for $p \geq 2$ and with $q < 2$ the H\"older conjugate of $p$, so the $2$-SGI obtained in the infinite-dimensional setting here is weaker in that $p$ must be larger than $2$, and $q$ cannot be less than $2$. The arguments for generalising the phase and interaction are identical. 

\begin{theorem}\label{sgiforN}
Let $\alpha > 0$ and $p > 3$. Define a local specification on $\bbH^{\bbZ^D}$ by the following potential:
\[ U_\Lambda^\omega(x_\Lambda) = \alpha \sum_{i \in \Lambda} \phi(x_i) + \sum_{i, j \in \Lambda} \beta_{ij} Q(x_i, x_j) + \sum_{i \in \Lambda, \, j \notin \Lambda} \beta_{ij} Q(x_i, \omega_j). \]
Suppose that:
\begin{enumerate}
    \item $\phi$ is a semibounded generalised polynomial in $N$ of order $p$ with leading term $\alpha N^p$, 
    \item $\abs{\beta_{ij}} \leq \beta$ for all $i, j \in \bbZ^D$, and
    \item $Q(x_i, x_j) = \sum_{\ell=1}^k c^\ell \varrho^\ell(x_i)^{\alpha_\ell} \delta^\ell(x_j)^{\beta_\ell}$ is a finite sum of $k \in \bbZ_{\geq 1}$ products of powers of homogeneous norms $(\varrho^\ell, \delta^\ell)_{\ell = 1, \cdots, k}$ with bounded subgradient and for $c^\ell \in \bbR$, $\alpha^\ell, \beta^\ell \in \{0\} \cup [1, \infty)$ and $0 \leq \alpha^\ell + \beta^\ell \leq p$.
\end{enumerate}
Then for all sufficiently small $\beta > 0$, there exists a unique Gibbs measure $\nu$ which satisfies the $2$-spectral gap inequality 
\[ \nu \abs{f - \nu f}^2 \leq C_{SG} \nu \abs{\nabla f}^2. \]
\end{theorem}

At this point, we remark that although the problem of replacing the phase with a different function (e.g. a different homogeneous norm not $d$ nor $N$) is more delicate as it requires the existence of a $U$-bound as in \bf{(C1)}, the issue of finding valid interactions is generally easier. For instance, we have not done so here to simplify the exposition, but we may easily adapt our analysis to the case where the generalised polynomials are given nonconstant smooth coefficients. In that case it is straightforward to see that it suffices the order $\alpha + \beta$ of the monomial $\psi \varrho^\alpha \delta^\beta$ satisfy $\alpha + \beta \leq p - 1 < p$ and that $\psi$ be bounded with bounded subgradient. (By the Leibniz rule there is an extra term $\varrho^\alpha \delta^\beta \nabla \varrho \lesssim \varrho^\alpha \delta^\beta$.) Nonetheless for the exposition given here the choice is since the ``typical" objects of interest are usually homogeneous norms. 

%{\color {red} Working Remark: A comment about a (homogeneous) polynomial (w.r.t. coordinates) would be interesting here.}
\begin{remark}
Interactions given in terms of a (combination) of semibounded homogeneous polynomial would be interesting here. We remark that generally one can represent them as potentials given in terms of some homogeneous norms plus some correction (which could be possibly rearranged to a function of some other homogeneous norms). A nice example of such a reduction is provided in \cite{DaWaZe22} where it was proved that a sum of squares of harmonic functions for the sub-Laplacian in Heisenberg group can be expressed in terms of powers of the corresponding Kaplan norm.
\end{remark}

\begin{remark}
We remark that given a phase $\phi$ which depends on a smooth homogeneous norm
one can improve the properties of the one site measure by adding a potential with singularity in the horizontal direction \cite{DaZe21b}.
\end{remark}

\section{Generalisations to other nilpotent Lie groups}

The method presented in this paper to prove $q$-SGIs can be extended to a wider class of spaces not and need not be restricted to the Heisenberg group. For instance, the results can be used to prove $q$-SGIs for Gibbs measures on $\bbR^n$. It can be shown using \cite[Theorem~2.4]{HeZe09} that the $U$-bound \eqref{uboundford} for the Carnot-Carath\'eodory distance $d$ and the sub-gradient $\nabla_\bbH$ is entirely analogous in the Euclidean case, in particular it holds with the Euclidean norm $\abs{x}$ replacing $d$ and the Euclidean gradient $\nabla_{\bbR^n}$ replacing $\nabla_\bbH$, which extends the results of \cite[Theorem~7.2]{BoZe05} in analogue with Theorem \ref{sgiford0}. This is since the main ingredients in obtaining the $U$-bound is firstly the fact $d$ satisfies the eikonal equation almost everywhere, a property shared with the Euclidean norm, and secondly the property $d\Delta_\bbH d$ is bounded above by some constant $K$ in the sense of distributions, a property which the Euclidean norm satisfies with $K = n-1$. 

\begin{theorem}\label{sgiEuclidean}
Let $\alpha > 0$ and $p \geq 2$. Define a local specification on $(\bbR^n)^{\bbZ^D}$ by the following potential:
\[ U_\Lambda^\omega(x_\Lambda) = \alpha \sum_{i \in \Lambda} \abs{x_i}^p + \sum_{i, j \in \Lambda} \beta_{ij} Q(x_i, x_j) + \sum_{i \in \Lambda, \, j \notin \Lambda} \beta_{ij} Q(x_i, \omega_j) \]
with $\abs{\beta_{ij}} < \beta$ for all $i, j \in \bbZ^D$, and for
\[ Q(x_i, x_j) = \sum_{k=1}^N \abs{x_i}^{r_k}\abs{x_j}^{s_k} \]
a finite sum of $N \in \bbZ_{\geq 1}$ monomials in $\abs{x_i}, \abs{x_j}$ of degree at most $p$ with nonnegative exponents $r_k, s_k$ for $k = 1, 2, \cdots, N$ belonging to $\{0\} \cup [1, \infty)$. Then for $q$ H\"older conjugate to $p$ and for all sufficiently small $\beta > 0$, there exists a unique Gibbs measure $\nu$ which satisfies the $q$-spectral gap inequality
\[ \nu\abs{f - \nu f}^q \leq C_{SG}\nu\abs{\nabla f}^q \]
for $f \in W^{1, q}(\nu)$.
\end{theorem}
% removed definition of \abs{\nabla f}^q; it appears already

The reason is because in our proof we made no use of any particular properties of the group; thus one could go through the proof and interpret $\nabla$ as the Euclidean gradient everywhere, that is with $X^1, X^2$ replaced by the trivial vector fields $\partial_1, \partial_2, \partial_3$, and in general with $\nabla$ any finite collection of vector fields on $\bbR^n$. That being said, one should not expect the $U$-bounds to come for free, and in general one should certainly expect the $U$-bound ingredients $\nabla$, $\phi$, and $\eta$, be intimately related. The class of nilpotent Lie groups in this regard are our typical examples because their structure gives rise to a rich family of nicely behaved examples. 

\begin{acknowledgements}
The first author is supported by the London Mathematical Society Early Career Fellowship.
\end{acknowledgements}

\printbibliography

\end{document}